\documentclass[twoside,english]{article}
\usepackage[T1]{fontenc}
\usepackage[latin9]{luainputenc}
\usepackage{geometry}
\usepackage{babel}
\usepackage{amsmath}
\usepackage{amsthm}
\usepackage{amssymb}
\usepackage[unicode=true,pdfusetitle,
 bookmarks=true,bookmarksnumbered=false,bookmarksopen=false,
 breaklinks=false,pdfborder={0 0 1},backref=false,colorlinks=false]
 {hyperref}

\makeatletter
\numberwithin{equation}{section}
\numberwithin{figure}{section}
\theoremstyle{plain}
\newtheorem{thm}{\protect\theoremname}[section]
\theoremstyle{remark}
\newtheorem{rem}[thm]{\protect\remarkname}
\theoremstyle{plain}
\newtheorem{lem}[thm]{\protect\lemmaname}
\theoremstyle{plain}
\newtheorem{cor}[thm]{\protect\corollaryname}
\theoremstyle{definition}
\newtheorem{defn}[thm]{\protect\definitionname}

\@ifundefined{date}{}{\date{}}
\usepackage[nottoc,notlot,notlof]{tocbibind}

\makeatother

\providecommand{\corollaryname}{Corollary}
\providecommand{\definitionname}{Definition}
\providecommand{\lemmaname}{Lemma}
\providecommand{\remarkname}{Remark}
\providecommand{\theoremname}{Theorem}

\begin{document}
\title{Classical and weak solutions to local first-order mean field games
through elliptic regularity }
\author{Sebastian Muñoz}
\maketitle
\begin{abstract}
We study the regularity and well-posedness of the local, first-order
forward--backward mean field games system, assuming a polynomially
growing cost function and a Hamiltonian of quadratic growth. We consider
systems and terminal data that are strictly monotone in the density
and study two different regimes depending on whether there exists
a lower bound for the running cost function. The work relies on a
transformation due to P.-L. Lions, which gives rise to an elliptic
partial differential equation with oblique boundary conditions, that
is strictly elliptic when the coupling is unbounded from below. In
this case, we prove that the solution is smooth. When the problem
is degenerate elliptic, we obtain existence and uniqueness of weak
solutions analogous to those obtained by P. Cardaliaguet and P.J.
Graber for the case of a terminal condition that is independent of
the density. The weak solutions are shown to arise as viscous limits
of classical solutions to strictly elliptic problems.
\end{abstract}
\textit{\small{}MSC: 35Q89, 35B65, 35J66, 35J70. }{\small\par}

\noindent \textit{\small{}Keywords: quasilinear elliptic equations;
oblique derivative problems; Bernstein method; non-linear method of
continuity; Hamilton-Jacobi equations; vanishing viscosity limit.}{\small\par}

\global\long\def\xo{\bar{x}}%

\global\long\def\tr{\text{{Tr}}}%

\global\long\def\osc{\text{\text{{osc}}}}%

\global\long\def\utilde{\tilde{u}}%

\global\long\def\diag{\text{{diag}}}%

\global\long\def\weak{\overset{\ast}{\rightharpoonup}}%

\global\long\def\intQ{\int\int_{Q_{T}}}%

\global\long\def\intO{\int_{\mathbb{T}^{d}}}%

\global\long\def\QT{\overline{Q_{T}}}%

\tableofcontents{}

\section{Introduction\label{sec:Introduction}}

The purpose of this paper is to study the well-posedness of the first-order
mean field games system (MFG for short) with a local coupling:

\begin{equation}
\tag{MFG}\begin{cases}
-u_{t}+H(x,D_{x}u)=f(x,m(x,t)) & (x,t)\in Q_{T}=\mathbb{T}^{d}\times(0,T),\\[1pt]
m_{t}-\textrm{div}(mD_{p}H(x,D_{x}u))=0 & (x,t)\in Q_{T},\\[4pt]
m(x,0)=m_{0}(x),\;u(x,T)=g(x,m(x,T)) & x\in\mathbb{T}^{d},
\end{cases}\label{eq:mfg}
\end{equation}
\renewcommand*{\theHequation}{nota1.\theequation}where $H:\mathbb{T}^{d}\times\mathbb{R}^{d}\rightarrow\mathbb{R}$
is a strictly convex Hamiltonian of quadratic growth, $f,g:\mathbb{T}^{d}\times[0,\infty)\rightarrow[-\infty,\infty)$
are strictly increasing in their second variable $m$, $f$ has polynomial
growth in $m$, and $m_{0}$ is a strictly positive probability density.
As is standard, we work on the flat $d$-dimensional torus $\mathbb{T}^{d}=\text{\ensuremath{\mathbb{R}^{d}/\mathbb{Z}^{d}}}$
to avoid additional technicalities with spatial boundary conditions.

MFG were introduced by Lasry and Lions \cite{LasryLions,Lions}, and
at the same time, in a particular setting, by Huang, Malhamé, and
Caines \cite{caines}. They are non-cooperative differential games
with infinitely many players, in which the players find an optimal
strategy by observing the distribution of the others. When the game
is completely deterministic, such games are typically modeled by the
system (\ref{eq:mfg}), which has been successfully studied in the
case where the function $g$ is independent of the density $m$, and
a complete theory of weak solutions has been obtained through variational
methods by Cardaliaguet, Graber, Porretta, and Tonon \cite{Cardalaguiet,CardaliaguetGraber,CardaliaguetGraberPorrettaTonon,Graber}.

Our two main contributions are, proving well-posedness when the terminal
condition is strictly increasing with respect to $m$, and the attainment
of classical solutions under the additional assumption that $f(\cdot,0)\equiv-\infty$.
When the latter blowup condition does not hold, we obtain weak solutions
that are in line with the variational theory, and they are shown to
enjoy higher regularity than in the case $g_{m}\equiv0$, by virtue
of the strict monotonicity of $g$.

The precise statements of our main results are as follows. We refer
to Section \ref{sec:Assumptions-and-general} for the exact assumptions
\hyperref[eq:M1]{(M)}, \hyperref[eq:strconv]{(H)}, \hyperref[eq:f polynomial growth]{(F)},
\hyperref[eq: gx control]{(G)}, (\ref{eq:SE}), and (\ref{eq:DE}),
to Section \ref{sec:Weak-solutions} for the definition of a weak
solution, and to the \hyperref[subsec:Notation]{notation} subsection
for the meaning of the function spaces mentioned in the theorems below.
\begin{thm}
\label{thm:smoothsols}Let $0<\alpha<1$, and assume that \hyperref[eq:M1]{(M)},
\hyperref[eq:strconv]{(H)}, \hyperref[eq:f polynomial growth]{(F)},
\hyperref[eq: gx control]{(G)}, and (\ref{eq:SE}) hold. Then there
exists a unique classical solution $(u,m)\in C^{3,\alpha}(\overline{Q_{T}})\times C^{2,\alpha}(\overline{Q_{T}})$
to (\ref{eq:mfg}).
\end{thm}

\begin{thm}
\label{thm:weaksols}Assume that \hyperref[eq:M1]{(M)}, \hyperref[eq:strconv]{(H)},
\hyperref[eq:f polynomial growth]{(F)}, \hyperref[eq: gx control]{(G)},
and (\ref{eq:DE}) hold. Then the following is true:
\begin{enumerate}
\item[(i)] There exists a weak solution $(u,m)\in(BV(Q_{T})\cap L^{\infty}(Q_{T}))\times(C([0,T],H^{-1}(\mathbb{T}^{d}))\cap L^{\infty}(Q_{T}))$
to (\ref{eq:mfg}). 
\item[(ii)] The solution $(u,m)$ is the a.e. limit, as $\epsilon\rightarrow0$,
of solutions $(u^{\epsilon},m^{\epsilon})\in C^{3,\alpha}(\overline{Q_{T}})\times C^{2,\alpha}(\overline{Q_{T}})$
to MFG systems satisfying (\ref{eq:SE}). Furthermore, $(u^{\epsilon}(\cdot,T),m^{\epsilon}(\cdot,T))\rightarrow(u(\cdot,T),m(\cdot,T))$
a.e. in $\mathbb{T}^{d}$.
\item[(iii)] If $(u',m')$ is another weak solution to (\ref{eq:mfg}), then $m=m'$
a.e. in $Q_{T}$, and $u=u'$ a.e. in $\{m>0\}$. Moreover, $m(\cdot,T)=m'(\cdot,T)$,
$u(\cdot,T)=u'(\cdot,T)$, and $u(\cdot,0)=u'(\cdot,0)$ a.e. in $\mathbb{T}^{d}$.
\end{enumerate}
\end{thm}

Despite the connections with the variational theory, we do not use
variational methods. Instead, we follow the ideas of Lions and his
work on the so-called planning problem, where the initial and terminal
densities $m(\cdot,0)$ and $m(\cdot,T)$ are prescribed \cite{LasryLions,Lions}.
It was first observed by Lions that if, for each fixed $x\in\mathbb{T}^{d},$
$f^{-1}(x,\cdot)$ is the inverse function of $f(x,\cdot)$, it is
possible to formally eliminate the variable $m$ from the system.
This transforms the problem into a second order quasilinear elliptic
equation with a non-linear oblique boundary condition which, in the
special case where $D_{x}H,D_{x}f,D_{x}g\equiv0$, may be written
as follows (see Section \ref{sec:Assumptions-and-general} for the
general setting):
\begin{equation}
\begin{cases}
-u_{tt}-\tr((D_{p}H\otimes D_{p}H+\chi(-u_{t}+H)D_{pp}^{2}H)D_{xx}^{2}u)+2D_{p}H\cdot D_{x}u_{t}=0 & \text{in }Q_{T},\\[1pt]
-u_{t}+H-f(m_{0})=0 & \text{on }\mathbb{T}^{d}\times\{t=0\},\\
-g(f^{-1}(-u_{t}+H))+u=0 & \text{on }\mathbb{T}^{d}\times\{t=T\},
\end{cases}\label{eq:ellip special}
\end{equation}
where the function $\chi(w)$ is defined by
\[
\chi(w)=f^{-1}(w)f_{m}(f^{-1}(w)).
\]
We emphasize the fact that, while (\ref{eq:ellip special}) is an
elliptic second order problem, the original system (\ref{eq:mfg})
is of first order and, in particular, it models a game with no diffusion. 

Our approach to obtain classical solutions when (\ref{eq:ellip special})
is strictly elliptic was developed by Lions, who applied, in his lectures
at Collège de France, the following strategy for finding regular solutions
to the planning problem. Viewed as a quasilinear elliptic equation
with a non-linear boundary condition, the problem can be tackled with
classical methods from the field of a priori estimates: specifically,
maximum principle techniques and the Bernstein method to obtain bounds
on the solution and its gradient, the application of classical estimates
to bound the Hölder norm of the gradient up to the boundary, and soft
functional analytic tools to attain the classical solutions.

In order to study the general MFG system, even when (\ref{eq:ellip special})
happens to be degenerate elliptic and the solutions are expected to
be discontinuous, our strategy is to first obtain smooth solutions
in the strictly elliptic case, and to subsequently find the weak solution
as a viscous limit of strictly elliptic problems. The success of this
approach is based on the fact that, once smooth solutions are known
to exist, every a priori estimate that is independent of the ellipticity
constant can be used as a source of compactness and regularity for
the limit. Our a priori estimates are supplemented by energy computations
based on the Lasry-Lions monotonicity procedure, which is the canonical
method for obtaining integral bounds and proving uniqueness in MFG
systems.

To identify and motivate the condition that determines the strict
or degenerate ellipticity of the system, we remark that the determinant
corresponding to the elliptic equation in (\ref{eq:ellip special})
becomes zero precisely as $\chi=mf_{m}\rightarrow0$. This is in accordance
with the heuristic principle that the regularity of $u$ is lost in
regions where there are few to no players (no information), as well
as when the cost fails to be strictly monotone (concentration blowup).
Because, as is standard, $f$ is assumed to grow at least logarithmically
as $m\rightarrow\infty$, this degeneracy can only happen as $m\rightarrow0$.
In the absence of diffusion, for the strict positivity of $m$ to
be preserved, we expect to have a very strong incentive for the players
to navigate through regions of low density. With these considerations
in place, we will classify the system (\ref{eq:mfg}) as being strictly
elliptic precisely when $f$ has a singularity at $m=0$, and as degenerate
elliptic otherwise.

It should be noted that, for the stationary problem, classical solutions
were obtained in \cite{Evans}, for the case where $f=\log m$, and
in \cite{Gomes-1}, for the case where $H(x,p)=\frac{1}{2}|p|^{2}-V(x)$,
under a small-oscillation assumption. For second order systems with
a (possibly) degenerate diffusion and a density-independent terminal
condition, the variational theory was extended in \cite{CardaliaguetGraberPorrettaTonon},
where it was shown (compare with Theorem \ref{thm:weaksols}) that
the weak solutions to the first order problem arise as viscous limits
of weak solutions to second order MFG systems. Finally, the most general
result for weak solutions to the second order problem is due to Porretta
\cite{Porretta}, and, unlike \cite{CardaliaguetGraberPorrettaTonon},
it does not use variational methods.

The content and structure of the paper are described as follows. Section
\ref{sec:Assumptions-and-general} explains the general setting and
assumptions that will be used, followed by the statements of the preliminary
results from the classical literature on quasilinear elliptic equations
and oblique derivative problems that will be used to prove existence
of classical solutions. 

In Section \ref{sec:A-priori-estimates}, we obtain all the necessary
a priori estimates for the strictly elliptic problem. The main results,
which deal with the system in full generality, are summarized in Theorem
\ref{thm:Full C^1 a priori bound}. We also establish a minor variant,
in the special case when the $x$ dependence has a simple structure,
that is, when $H(x,p)-f(x,m)\equiv H(p)-f(m)-V(x)$, in Theorem \ref{thm:V(x) a priori C^1}.
This result states that, with this structural assumption, it is not
necessary to require $f$ to grow at most polynomially in $m$, allowing
for examples such as $f(m)=e^{m}+\log m$. Section \ref{subsec:Estimates sol}
contains the $L^{\infty}$--bounds on the solution $u$, as well
as two-sided bounds for the terminal density $m(\cdot,T)$, obtained
through maximum principle methods. These methods exploit the fact
that the strict monotonicity property of $g$ is equivalent to the
linearized version of the problem (\ref{eq:ellip special}) having
an oblique boundary condition, which is of Robin type in the upper
component of $\partial Q_{T}$. In Section \ref{subsec:Estimates grad},
the gradient bound is obtained by means of the Bernstein method. To
deal with the asymmetry between the space and time derivatives in
(\ref{eq:ellip special}), it is necessary to first get a precise
bound for $u_{t}$ in terms of the space gradient, utilizing the a
priori lower bound on $m(\cdot,T)$ and the maximum principle. This,
in turn, provides a ``conditional'' a priori lower bound for $m$,
namely a lower bound that holds exclusively at points $(x,t)$ where
the function $H(x,D_{x}u)$ is close to its maximum value. The conditional
nature of this bound, as well as the structure of (\ref{eq:ellip special})
in its fully general form, require a non-conventional choice of an
auxiliary function of the space-time gradient.

Section \ref{sec:Classical-solutions} deals with the existence of
classical solutions for the strictly elliptic problem, including the
proof of Theorem \ref{thm:smoothsols}. The corresponding variant
for the case of a fast-growing $f$ is presented in Theorem \ref{thm:V(x) smoothsols}.
It is first explained how a classical result from the theory of oblique
derivative problems, due to G.M. Lieberman \cite{Lieberman}, immediately
yields an a priori Hölder estimate for $Du$ up to the boundary in
terms of the $L^{\infty}$--bounds on $u$ and $Du$. Existence is
then proved through an application of the non-linear method of continuity,
the classical Schauder estimates for the linear oblique derivative
problem, and a variant of a convergence theorem of R. Fiorenza \cite{Fiorenza,Fiorenza-1,Lieberman solvability}.

In Section \ref{sec:Weak-solutions}, we develop the weak theory for
the degenerate elliptic problem, and obtain the proof of Theorem \ref{thm:weaksols}.
It is first established that, for strictly elliptic problems, there
exists an upper bound for the density that is independent of any lower
bounds on $m(\cdot,T)$. After deriving some necessary energy estimates
and defining an $\epsilon$-perturbation of the coupling $f$ that
makes the problem strictly elliptic, the solution is obtained as the
limit when $\epsilon\rightarrow0$ of the corresponding smooth solutions.
It is also proved, in Theorem \ref{thm:Lip degenerate}, that when
the data is independent of the space variable, the value function
$u$ and the terminal density $m(\cdot,T)$ are globally Lipschitz
continuous.
\begin{rem}
We mention here some related work that was released after this paper.
In \cite{Munoz}, the author showed existence of classical solutions
for the so-called extended MFG, a generalization of (\ref{eq:mfg})
introduced by Lions and Souganidis \cite{LionsSoug}, having a fully
general continuity equation, and a non-separated Hamiltonian, namely
$H=H(x,p,m)$, with arbitrary superlinear growth. In particular, classical
solutions were obtained for first order MFG with congestion. As for
weak solutions to (\ref{eq:mfg}), the most general result to date
was obtained by Cardaliaguet and Porretta \cite{CardaliaguetPorretta},
where the solution is obtained as a vanishing viscosity limit to the
weak solutions from \cite{Porretta}.
\end{rem}

\subsection*{Notation\label{subsec:Notation}}

Let $n,k\in\mathbb{N}.$ Given $x,y\in\mathbb{R}^{n},$ $x$ and $y$
will always be understood to be row vectors, and their scalar product
$xy^{T}$ will be denoted by $x\cdot y$. For any bounded set $\Omega$,
with $\Omega\subset Q_{T}$, $\Omega\subset\mathbb{T}^{d}$, or $\Omega\subset[0,T]$,
and $0\leq\alpha<1$, $C^{k,\alpha}(\Omega$), refers to the space
of $k$ times differentiable real-valued functions with $\alpha$--Hölder
continuous $k$th order derivatives, and, for $u\in C^{0,\alpha}(\Omega)$,
the Hölder semi-norm of $u$ will be denoted by $[u]_{\alpha,\Omega}.$
Similarly, if $H^{-1}(\mathbb{T}^{d})$ denotes the dual space of
the Sobolev space $H^{1}(\mathbb{T}^{d})$, the space of $H^{-1}(\mathbb{T}^{d})$--valued
$\alpha$--Hölder continuous functions $C^{0,\alpha}([0,T];H^{-1}(\mathbb{T}^{d}))$
is equipped with the Hölder semi-norm $[\cdot]_{\alpha,[0,T],H^{-1}}$.
For functions $\Phi(x,t,z,p,s)\in C^{0}(Q_{T}\times\mathbb{R}\times\mathbb{R}^{d+1})$,
where typically $(x,t,z,p,s)=(x,t,u(x,t),D_{x}u,u_{t})$, the conventions
$\xo\equiv(x,t)$ and $q\equiv(p,s)$ will always be in place. The
notation $Du$, $D\Phi$ will always refer to the full gradient in
all variables, so that, for instance $Du=D_{\xo}u=(D_{x}u,u_{t})$,
and $D\Phi=(D_{\xo}\Phi,\Phi_{z},D_{q}\Phi$). For $(x,t)\in\partial Q_{T}$,
$\nu(x,t)=\pm(0,0,\ldots,1)$ denotes the outward pointing unit normal
vector. We write $C=C(K_{1},K_{2},\ldots,K_{M})$ for a positive constant
$C$ depending monotonically on the non-negative quantities $K_{1},\ldots,K_{M}.$
We also define, for $K>0$, and any set $V,$ $V_{K}=\{(y,z,q)\in V\times\mathbb{R\times}\mathbb{R}^{d+1}:|z|+|q|\leq K\}$.
We write $C^{k}(\overline{Q_{T}})^{*}$ for the dual space of $C^{k}(\overline{Q_{T}})$.
In particular, $C^{0}(\overline{Q_{T}})^{*}$ is the space of finite
signed Borel measures on $\overline{Q_{T}}$, and $C^{\infty}(\overline{Q_{T}})^{*}$
is the space of distributions. Moreover, BV$(Q_{T})$ is the space
of functions of bounded variation, that is, the space of $L^{1}(Q_{T})$
functions such that their distributional derivatives are elements
of $C^{0}(\overline{Q_{T}})^{*}$, and $L_{+}^{\infty}(Q_{T}$) consists
of the functions $m\in L^{\infty}(Q_{T}$) such that $m\geq0$ almost
everywhere (a.e. for short) in $Q_{T}$. Finally, for $m\in L_{+}^{\infty}(Q_{T})$,
$L_{m}^{2}(Q_{T})$ consists of the functions $v$ such that $|v|^{2}m\in L^{1}(Q_{T})$. 

\section{Assumptions and general setting \label{sec:Assumptions-and-general}}

\subsection{The MFG system as an elliptic problem}

We now present the general elliptic formulation of the MFG system.
As explained in the previous section, it is an equivalent problem
satisfied by $u$, whenever the pair $(u,m)=(u,f^{-1}(\cdot,-u_{t}+H(\cdot,D_{x}u))\in C^{2}(\overline{Q_{T}})\times C^{1}(\overline{Q_{T}})$
is a classical solution to (\ref{eq:mfg}). It is obtained after eliminating
$m$ from the system, and it consists of a quasilinear elliptic equation
with a non-linear oblique boundary condition,
\begin{equation}
\tag{Q0}\begin{cases}
Qu=-\tr(A(x,Du)D^{2}u)+b(x,Du)=0 & \text{in }Q_{T},\\[1pt]
Nu=B(x,t,u,Du)=0 & \text{on }\partial Q_{T},
\end{cases}\label{eq:ellip}
\end{equation}
\renewcommand*{\theHequation}{nota2.\theequation}where, for all $(x,t,z,p,s)\in\overline{Q_{T}}\times\mathbb{R}\times\mathbb{R}^{d+1}$,\vspace*{-0.3cm}

\begin{align}
\tag{Q1}A(x,p,s) & =(D_{p}H,-1)\otimes(D_{p}H,-1)+\chi(x,-s+H(x,p))\begin{pmatrix}D_{pp}^{2}H(x,p) & 0\\
0 & 0
\end{pmatrix},\label{eq:matrix}
\end{align}
\renewcommand*{\theHequation}{nota3.\theequation}\vspace*{-0.3cm}
\begin{multline}
\tag{Q2}b(x,p,s)=-D_{x}H(x,p)\cdot D_{p}H(x,p)+D_{x}f(x,f^{-1}(x,-s+H(x,p)))\cdot D_{p}H(x,p)\\
-\chi(x,-s+H(x,p))\tr(D_{xp}^{2}H(x,p)),\label{eq:first order term}
\end{multline}
\begin{align}
\tag{B1}B(x,0,z,p,s)= & -s+H(x,p)-f(x,m_{0}(x)),\;B(x,T,z,p,s)=-g(x,f^{-1}(x,-s+H(x,p)))+z,\label{eq:boundary}
\end{align}
\renewcommand*{\theHequation}{nota4.\theequation}with the function
$\chi(x,w)$ being defined by
\[
\chi(x,w)=f^{-1}(x,w)f_{m}(x,f^{-1}(x,w)).
\]
We remark that the matrix $A$ is clearly non-negative, and since
$\text{\;\ensuremath{\det}}(A)=\chi^{d}\det D_{pp}^{2}H,$ the condition
for degeneracy is $\chi=mf_{m}=0$. For future use, we set

\[
h(x,w)=\sqrt{\chi(x,w)}.
\]

\subsection{Assumptions}

We now state the main assumptions \hyperref[eq:M1]{(M)}, \hyperref[eq:strconv]{(H)},
\hyperref[eq:f polynomial growth]{(F)}, \hyperref[eq: gx control]{(G)},
and \hyperref[eq:SE]{(E)}, that will be in place throughout the paper,
except when explicitly stated. Assumption \hyperref[eq:SE]{(E)} on
the ellipticity of the system contains the mutually exclusive possibilities
(\ref{eq:SE}) and (\ref{eq:DE}), and it will always be made clear
which of the two is in place. For the theory of weak solutions, the
differentiability assumptions on the data can naturally be weakened
through standard approximation arguments, but in the interest of clarity
such matters will not be considered. Throughout the assumptions, the
quantities $C_{0}>0$ and $0\leq\tau<1$ are fixed constants.
\begin{enumerate}
\item[(M)] (Assumptions on $m_{0}$) The initial density $m_{0}$ satisfies
\begin{equation}
\tag{M1}m_{0}\in C^{4}(\mathbb{T}^{d}),\;m_{0}>0,\text{ and }\int_{\mathbb{T}^{d}}m_{0}=1.\label{eq:M1}
\end{equation}
\renewcommand*{\theHequation}{notag5.\theequation}
\item[(H)] (Assumptions on $H$) The functions $H,\;D_{p}H,\;D_{pp}^{2}H$ are
four times continuously differentiable, and the following quadratic
growth and uniform convexity conditions hold:
\begin{equation}
\tag{H1}\frac{1}{C_{0}}I\leq D_{pp}^{2}H(x,p)\leq C_{0}I,\label{eq:strconv}
\end{equation}
\renewcommand*{\theHequation}{notag.\theequation}
\begin{equation}
\tag{H2}D_{p}H(x,p)\cdot p\geq2H(x,p)-C_{0},\label{qg1}
\end{equation}
\renewcommand*{\theHequation}{notag2.\theequation}
\begin{equation}
\tag{H3}|D_{ppp}^{3}H(x,p)|\leq C_{0}(1+|p|)^{-1},\label{eq:qg2}
\end{equation}
\renewcommand*{\theHequation}{notag3.\theequation}for all $(x,p)\in\mathbb{T}^{d}\times\mathbb{R}^{d}$.
The space oscillation of $H$ is at most subquadratic in $p$, namely
\begin{align}
\tag{HX}|D_{xxp}^{3}H| & \leq C(1+|p|)^{\tau},\;|D_{xpp}^{3}H|\leq C_{0}(1+|p|)^{\tau-1}.\label{eq:Hxxp, Hxpp bound}
\end{align}
\renewcommand*{\theHequation}{notag4.\theequation}
\item[(F)] (Assumptions on $f$) The continuous function $f:\mathbb{T}^{d}\times[0,\infty)\rightarrow[-\infty,\infty)$
is four times continuously differentiable on $\mathbb{T}^{d}\times(0,\infty)$
and strictly increasing in the second variable, with $f_{m}>0$. $f$
grows polynomially as $m\rightarrow\infty$, in the sense that its
growth is at least of degree zero, namely
\begin{equation}
\tag{F1}\liminf_{x\in\mathbb{T}^{d},\;m\rightarrow\infty}mf_{m}(x,m)>0,\label{eq:f polynomial growth}
\end{equation}
and its derivative $f_{m}$ satisfies a polynomial bound $|mf_{mm}|\leq C_{0}f_{m},$
which can be equivalently expressed in terms of $\chi(x,w)$ as
\begin{equation}
\tag{F2}|\chi_{w}|\leq C_{0}.\label{growth chi}
\end{equation}
The space derivative of $f$ satisfies the same polynomial bound,
\begin{equation}
\tag{FX1}|m(D_{x}f)_{m}|\leq C_{0}|D_{x}f|,\label{eq: fmx bound}
\end{equation}
\renewcommand*{\theHequation}{notag8.\theequation}as well as the
control
\begin{align}
\tag{FX2}|D_{x}f|,|D_{xx}^{2}f|\leq & C_{0}(1+|f|^{\tau/2}+|mf_{m}|^{(1+\tau)/2}).\label{eq:fx, fxx bound}
\end{align}
\renewcommand*{\theHequation}{not.\theequation}
\item[(G)] (Assumptions on $g$) The continuous function $g:\mathbb{T}^{d}\times[0,\infty)\rightarrow[-\infty,\infty)$
is four times continuously differentiable on $\mathbb{T}^{d}\times(0,\infty)$
and strictly increasing in the second variable, with $g_{m}>0$. The
control required for its space oscillation is that, for each $x\in\mathbb{T}^{d}$,
\begin{align}
\tag{GX} & \lim_{m\rightarrow\infty}g(x,m)=\sup_{\mathbb{T}^{d}\times[0,\infty)}g,\text{ and }g(x,0)=\inf_{\mathbb{T}^{d}\times[0,\infty)}g,\label{eq: gx control}
\end{align}
\renewcommand*{\theHequation}{notag10.\theequation}
\item[(E)]  (Ellipticity of the system) One of the following conditions holds:
\begin{equation}
\text{\tag{SE}the system (MFG) is \emph{strictly elliptic}, in the sense that }f(\cdot,0)\equiv-\infty,\label{eq:SE}
\end{equation}
\renewcommand*{\theHequation}{not1.\theequation}or
\begin{equation}
\tag{DE}\text{the system (MFG) is \emph{degenerate elliptic}, in the sense that }f(\cdot,0)>-\infty.\label{eq:DE}
\end{equation}
\renewcommand*{\theHequation}{not2\theequation}In the case of (\ref{eq:DE}),
since the density is not expected to be strictly positive, we assume
that $g(\cdot,0)>-\infty$.
\end{enumerate}
A few comments should be made about the assumptions on the spatial
oscillation. First, we remark that the subquadratic growth assumption
(\ref{eq:Hxxp, Hxpp bound}) can be interpreted as requiring that
the purely quadratic part of $H$ is independent of $x$. Condition
(\ref{eq:fx, fxx bound}), on the other hand, can be interpreted as
being dual to (\ref{eq:Hxxp, Hxpp bound}). Indeed, heuristically,
since $f$ is assumed to have polynomial growth, $mf_{m}\approx f$,
and $f=-u_{t}+H\approx|p|^{2}$, so both conditions impose the same
polynomial growth bound in the variable $|p|$. We consider now the
assumption (\ref{eq: gx control}) on the $x$--oscillation of $g$.
When $g$ is bounded, the first (resp. second) condition in (\ref{eq: gx control})
corresponds to a purely qualitative control on $|D_{x}g|$ that becomes
stricter as $m\rightarrow\infty$ (resp. $m\rightarrow0^{+})$. From
the modeling point of view, it can be interpreted as the requirement
that extremely crowded regions (resp. nearly empty regions) have roughly
the same terminal value for the players. 
\begin{rem}
For simplicity of the presentation, we observe that, up to increasing
the value of $C_{0}$, the following inequalities are trivial consequences
of (\ref{eq:strconv}), (\ref{eq:Hxxp, Hxpp bound}), and (\ref{eq:SE}),
and they will be used freely when pertinent.
\begin{equation}
\frac{1}{C_{0}}|p|^{2}-C_{0}\leq H(x,p)\leq C_{0}|p|^{2}+C_{0},|D_{p}H(x,p)|\leq C_{0}(1+|p|),\label{eq: H>=00003D|p|^2-C}
\end{equation}
\begin{equation}
|D_{x}H(x,p)|\leq C_{0}(1+|p|^{1+\tau}),\;|D_{xx}^{2}H(x,p)|\leq C_{0}(1+|p|^{1+\tau}),\label{eq: Hx, Hxx bound}
\end{equation}

\begin{equation}
|D_{xp}^{2}H(x,p)|\leq C_{0}(1+|p|)^{\tau},\label{eq: Hxp bound}
\end{equation}
\begin{align}
||\chi(\cdot,0)||_{C^{0}(\mathbb{T}^{d})}+||m_{0}||_{C^{1}(\mathbb{T}^{d})}+||f||_{C^{2}(\mathbb{T}^{d}\times[\min m_{0},\max m_{0}])}\leq C_{0}.\label{eq: old C1}
\end{align}
\end{rem}

\subsection{Preliminary results}

This subsection includes the classical results that will be required
in Section \ref{sec:Classical-solutions} to obtain the higher regularity
from a priori $C^{1}$ bounds. In this subsection only, it will not
be assumed that the problem (\ref{eq:ellip}) is explicitly given
by (\ref{eq:matrix}), (\ref{eq:first order term}), and (\ref{eq:boundary}),
but instead $(Q,N)$ will be a general pair of an elliptic quasilinear
operator and a fully non-linear boundary operator. In particular,
$A$ and $b$ will not necessarily be assumed to be independent of
$t$ and $u$. The first Theorem is the classical interior Hölder
gradient estimate for quasilinear equations, due to O. Ladyzhenskaya
and N. Uraltseva \cite[Lem. 2.1]{Lieberman}. 
\begin{thm}
\label{thm:Ladyzhenskaya interior}Let $u\in C^{2}(Q_{T})$ satisfy
$Qu=0$ in $Q_{T}$, with $A(x,t,z,q)\in C^{1}(Q_{T}\times\mathbb{R}\times\mathbb{R}^{d+1})$,
$b(x,t,z,q)\in C^{0}(Q_{T}\times\mathbb{R}\times\mathbb{R}^{d+1})$.
Suppose that $||u||_{C^{1}(Q_{T})}\leq K$, and that the constants
$\lambda_{K},\mu_{K}$ satisfy, in $Q_{T,K}$,
\begin{align}
A\geq\lambda_{K}I\;\;\text{ and \;\;}\mu_{K}\geq|A|+|DA|+|b|.\label{eq:interior estimate, ellipticity and coeff bound}
\end{align}
Then, for any $V\subset\subset Q_{T}$, there exist constants $C=C(K,\mu_{K}/\lambda_{K},\emph{dist}(V,\partial Q_{T})^{-1})$
and $\gamma=\gamma(K,\mu_{K}/\lambda_{K})$, such that
\[
[Du]_{\gamma,V}\leq C.
\]
\end{thm}

Next is the following local boundary Hölder estimate for the gradient
in oblique problems, due to Lieberman \cite[Lem. 2.3]{Lieberman}.
In Theorem \ref{thm:Lieberman oblique boundary}, the following definitions
are in place:
\begin{align*}
B & =\{(x,t)\in\mathbb{R}^{d+1}:|(x,t)|<1\},\;B^{+}=\{(x,t)\in B:t>0\},\\
B^{0} & =\{(x,t)\in B:t=0\},\;B'=\{(x,t)\in B^{+}:|(x,t)|<\frac{1}{3}\}.
\end{align*}

\begin{thm}
\label{thm:Lieberman oblique boundary}Let $u\in C^{2}(B^{+}\cup B^{0})$
solve $Qu=0$ in $B^{+}$, $Nu=0$ on $B^{0}$, with $A(x,t,z,q)\in C^{1}(\overline{B^{+}}\times\mathbb{R}\times\mathbb{R}^{d+1})$,
$b(x,t,z,q)\in C^{0}(\overline{B^{+}}\times\mathbb{R}\times\mathbb{R}^{d+1})$,
$B(x,t,z,q)\in C^{1}(B^{0}\times\mathbb{R}\times\mathbb{R}^{d+1}),\;D_{q}B(x,t,z,q)\in C^{1}(B^{0}\times\mathbb{R}\times\mathbb{\mathbb{R}}^{d+1}).$
Assume furthermore that (\ref{eq:interior estimate, ellipticity and coeff bound})
holds in $\overline{B_{K}^{+}}$, as well as, on $B_{K}^{0}$,
\begin{align}
\lambda_{K}\leq & -B_{s},\text{ and}\nonumber \\
\mu_{K}\geq & |D_{q}B|+|D_{z}B|+|D_{\xo}B|+|D_{qq}^{2}B|+|D_{qz}^{2}B|+|D_{q\xo}^{2}B|\label{eq:boundary muK}
\end{align}
Then there are constants $C$ and $\gamma$ depending only on $K$
and $\mu_{K}/\lambda_{K}$ such that, if $||u||_{C^{1}(B^{+}\cup B^{0})}\leq K$,
then 
\[
[Du]_{\gamma,B'}\leq C.
\]
\end{thm}

For the next theorem, which is the basic Schauder estimate for linear
oblique problems \cite[Thm. 6.30]{GilbargTrudinger}, we recall that
$\nu(x,t)=\pm(0,0,\ldots,1)$ denotes the outward pointing normal
vector at $(x,t)\in\partial Q_{T}$.
\begin{thm}
\label{thm:oblique Schauder}Assume that $u\in C^{2}(\overline{Q_{T}})$
solves the linear problem
\[
-\emph{\text{Tr}}(\tilde{A}(x,t)D^{2}u)=\eta_{1}(x,t)\text{ in }Q_{T},\;\;\;\tilde{B}(x,t)\cdot Du=\eta_{2}(x,t)\text{ on }\partial Q_{T},
\]
where 
\[
\tilde{A^{ij}},\eta_{1}\in C^{0,\alpha}(\overline{Q_{T}}),\;\;\tilde{B},\eta_{2}\in C^{1,\alpha}(\partial Q_{T}),\;\;\tilde{A}\geq\lambda I,\text{ and }\tilde{B}\cdot\nu\geq\lambda_{0}.
\]
 Then there exists $C=C(\frac{1}{\lambda},\frac{1}{\lambda_{0}},||\tilde{A}^{ij}||_{C^{0,\alpha}(\overline{Q_{T}})},||\tilde{B}||_{C^{1,\alpha}(\partial Q_{T})})$
such that
\[
||u||_{C^{2,\alpha}(\overline{Q_{T}})}\leq C(||u||_{C^{0}(\overline{Q_{T}})}+||\eta_{1}||_{C^{0,\alpha}(\overline{Q_{T}})}+||\eta_{2}||_{C^{1,\alpha}(\partial Q_{T})}).
\]
\end{thm}

The last result of this subsection is a variant of a convergence theorem
of Fiorenza, which is a basic tool for using the method of continuity
without the need of a priori second derivative estimates \cite[Lem. 2, Cor. 1]{Lieberman solvability}.
\begin{thm}
\label{thm:Fiorenza convergence}Let $0<\alpha,\gamma<1$. For each
$n\in\mathbb{N}$, let $u_{n}\in C^{2,\alpha}(\overline{Q_{T}})$
be a sequence of solutions to the quasilinear problems $Q_{n}u=0$,
$N_{n}u=0$, where, for $C,K,\gamma,\lambda,\lambda_{0}$ independent
of $n,$
\begin{align*}
Q_{n}u & =-\emph{\text{Tr}}(A_{n}(x,t,u,Du)D^{2}u)+b_{n}(x,t,u,Du),\;\;N_{n}u=B_{n}(x,t,u,Du),
\end{align*}
\vspace*{-1.2cm}

\begin{align*}
||A_{n}||_{C^{1}(\overline{Q}_{T,K})}+||b_{n}||_{C^{1}(\overline{Q}_{T,K})}+||B_{n}||_{C^{2}(\overline{Q}_{T,K})}+||D_{q}B_{n}||_{C^{2}(\overline{Q}_{T,K})} & \leq C,
\end{align*}
\vspace*{-0.67cm}
\begin{align*}
A_{n} & \geq\lambda I\text{ in }\overline{Q}_{T,K},\;\text{ and }\;D_{q}B_{n}\cdot\nu\geq\lambda_{0}\text{ in }\partial\overline{Q}_{T,K},
\end{align*}
\[
||u_{n}||_{C^{1+\gamma}(\QT)}\leq K,
\]
with $u_{n}\rightarrow u$ uniformly, and $(A_{n},$$b_{n},B_{n})\rightarrow(A,b,B)$
uniformly on $\overline{Q}_{T,K}$. Then $u_{n}\rightarrow u$ in
$C^{2,\alpha}(\overline{Q_{T}})$, and $u$ solves (\ref{eq:ellip}).
\end{thm}

\section{A priori estimates\label{sec:A-priori-estimates}}

In this section, we establish a priori estimates for the solution
and the gradient, in the case where (\ref{eq:mfg}) is strictly elliptic.
To account for the fact that the functions $f$ and $g$ depend on
the space variable, we will make extensive use of the continuous,
strictly increasing functions $f_{0},g_{0},f_{1},g_{1}:(0,\infty)\rightarrow\mathbb{R}$
defined by
\begin{align*}
f_{0}(m)=\min_{\mathbb{T}^{d}}f(\cdot,m),\;\;g_{0}(m)=\min_{\mathbb{T}^{d}}g(\cdot,m),\;\;f_{1}(m)=\max_{\mathbb{T}^{d}}f(\cdot,m),\;\;g_{1}(m)=\max_{\mathbb{T}^{d}}g(\cdot,m).
\end{align*}

\subsection{Estimates for the solution and the terminal density\label{subsec:Estimates sol}}

We first obtain a priori bounds for the $C^{0}$ norm of the solution
$u$. As a corollary, positive, two-sided bounds for the terminal
density are established. 
\begin{lem}
\label{lem:aprioriu}Assume that (\ref{eq:SE}) holds. Then, there
exists a constant $C=C(C_{0})$ such that for any solution $(u,m)\in C^{2}(\overline{Q_{T}})\times C^{1}(\overline{Q_{T}})$
of (\ref{eq:mfg}), and every $(x,t)\in\overline{Q_{T}},$ 
\begin{equation}
g_{0}f_{1}^{-1}(-C)-C(e^{CT}-e^{Ct})\leq u(x,t)\leq g_{1}f_{0}^{-1}(C)+C(e^{CT}-e^{Ct}).\label{aprioriu1}
\end{equation}
\end{lem}

\begin{proof}
The goal here is to modify $u$ into a function that necessarily achieves
its maximum at $\{t=T\}$, which is the region of the boundary where,
by the strict monotonicity of $g$, the boundary condition of (\ref{eq:ellip})
provides information about $u$. This requires some estimates for
the terms in (\ref{eq:first order term}). By (\ref{eq: Hxp bound})
and (\ref{growth chi}), 
\begin{equation}
|\chi(x,f)\tr(D_{xp}^{2}H(x,D_{x}u))|\leq C(1+|f|)(1+|D_{x}u|^{\tau}).\label{bound tr(Hpx)}
\end{equation}
Moreover, by (\ref{eq:fx, fxx bound}),
\begin{align}
|D_{x}f(x,m(x,t))\cdot D_{p}H(x,D_{x}u)| & \leq C(1+|f|^{(1+\tau)/2})(1+|D_{x}u|).\label{eq:}
\end{align}
Now, given $u,$ define the linear, uniformly elliptic operator $Q_{u}$
by
\[
Q_{u}v=-\tr(A(x,Du)D^{2}v).
\]
Notice that $Q_{u}u=-b(x,Du)$. Let $\zeta\in C^{2}([0,T])$ be a
function to be chosen later, and define
\[
v=u+\zeta(t),
\]
so that $v_{t}=u_{t}+\zeta'(t)$ and $D_{x}v=D_{x}u$. This yields,
by (\ref{eq: H>=00003D|p|^2-C}), (\ref{eq: Hx, Hxx bound}), (\ref{bound tr(Hpx)}),
and (\ref{eq:}),
\begin{align*}
Q_{u}v= & -\zeta''(t)+D_{x}H(x,D_{x}v)\cdot D_{p}H(x,D_{x}v)-D_{x}f(x,m)\cdot D_{p}H(x,D_{x}v)+\chi\tr(D_{xp}^{2}H(x,D_{x}v))\\
\leq & -\zeta''(t)+C(1+|D_{x}v|^{1+\tau})(1+|D_{x}v|)+C(1+|-v_{t}+H(x,D_{x}v)+\zeta'(t)|^{(1+\tau)/2})(1+|D_{x}v|)\\
 & +C(1+|-v_{t}+H(x,D_{x}v)+\zeta'(t)|)(1+|D_{x}v|^{\tau})\\
\leq & -\zeta''(t)+C(1+|D_{x}v|^{3}+|v_{t}|^{2})+C(1+|D_{x}v|)|\zeta'(t)|,
\end{align*}
where the constant $C$ increases in each line. Now, set $C_{1}=2C$
and fix $C_{1}$, still allowing $C$ to increase at each step. We
choose $\zeta(t)=\frac{k}{2C_{1}}(e^{2C_{1}t}-e^{2C_{1}T})$, where
$k>0$ is a parameter. Then, 
\[
\zeta'(t)=ke^{2C_{1}t},\;\zeta''(t)=2C_{1}|\zeta'(t)|,
\]
and, consequently, at any interior maximum point $(x,t)$ of $v$,
\begin{align*}
0\leq Q_{u}v\leq-\zeta''(t)+C_{1}(1+|\zeta'(t)|)\leq-C_{1}\zeta'(t)+C_{1}=-C_{1}ke^{2C_{1}t}+C_{1}\leq-C_{1}k+C_{1},
\end{align*}
which can only hold if $k\leq1$. Thus, if one chooses $k>1$, $v$
necessarily achieves its maximum value when $t=0$ or $t=T$. If this
happens at a point $(x,t)$ where $t=0$, then $u_{t}+\zeta'=v_{t}\leq0$,
$D_{x}u=D_{x}v=0$. Therefore,
\begin{align*}
-||H(\cdot,0)||_{C^{0}(\mathbb{T}^{d})} & \leq-v_{t}+H(x,0)=-u_{t}+H(x,D_{x}u)-\zeta'(0)=f(x,m_{0}(x,t))-\zeta'(0),
\end{align*}
implying that
\[
k=\zeta'(0)\leq f(x,m_{0}(x,t))+||H(\cdot,0)||_{C^{0}(\mathbb{T}^{d})}.
\]
Hence, taking $k>\max_{x\in\mathbb{T}^{d}}f(x,m_{0}(x))+||H(\cdot,0)||_{C^{0}(\mathbb{T}^{d})}$,
it follows that $v$ attains its maximum value at $t=T$. At this
point, $u_{t}+\zeta'(t)=v_{t}\geq0$, $D_{x}u=D_{x}v=0$, and, as
before,
\begin{align*}
||H(x,0)||_{C^{0}}\geq-v_{t}+H(x,0)=f(x,m(x,T))-\zeta'(T),
\end{align*}
which gives
\begin{align*}
f_{0}(m(x,T))\leq f(x,m(x,T))\leq\zeta'(T)+||H(\cdot,0)||_{C^{0}(\mathbb{T}^{d})}\leq ke^{2C_{1}T}+||H(\cdot,0)||_{C^{0}(\mathbb{T}^{d})}\leq C.
\end{align*}
Thus, since $u(x,T)=v(x,T)$, taking into account the surjectivity
of $f(x,\cdot)$,
\begin{align*}
\max v & =v(x,T)=g(x,m(x,T))\leq g(x,f_{0}^{-1}(C))\leq g_{1}(f_{0}^{-1}(C)).
\end{align*}
Finally, for arbitrary $(x,t)\in Q_{T},$
\begin{align*}
u(x,t) & =v(x,t)-\zeta(t)\leq g_{1}(f_{0}^{-1}(C))+C(e^{CT}-e^{Ct}).
\end{align*}
The lower estimate follows from a completely symmetrical argument.
\end{proof}
\begin{cor}
\label{cor:apriorim}Assume (\ref{eq:SE}), and let $C$ be the constant
from Lemma \ref{lem:aprioriu}. Then, for every $x\in\mathbb{T}^{d},$
\begin{align}
g_{1}^{-1}g_{0}f_{1}^{-1}(-C)\leq & m(x,T)\leq g_{0}^{-1}g_{1}f_{0}^{-1}(C),\label{eq:a priori m(T) inequality}
\end{align}
\end{cor}

\begin{proof}
From the first inequality in (\ref{aprioriu1}), for each $x\in\mathbb{T}^{d},$
\[
g_{0}f_{1}^{-1}(-C)\leq g(x,m(x,T)),
\]
and thus, by definition of $g_{1}$, 
\begin{equation}
g_{0}f_{1}^{-1}(-C)\leq g_{1}(m(x,T)).\label{eq:  cor a priori m (inter)}
\end{equation}
Observe that the application of $g_{1}^{-1}$ on both sides of (\ref{eq:  cor a priori m (inter)})
is possible because, by (\ref{eq: gx control}), the functions $g_{0}$
and $g_{1}$ have the same range. This yields the first inequality
in (\ref{eq:a priori m(T) inequality}). The second inequality is
obtained through the same reasoning.
\end{proof}
\begin{rem}
\label{rem:remark u, m(T) bound}A minor modification of the proof
of Lemma \ref{lem:aprioriu} shows that, when $H$, $f$, and $g$
are independent of $x$, the following sharper estimates hold:
\[
g(\min m_{0})+(f(\min m_{0})-H(0))(T-t)\leq u(x,t)\leq g(\max m_{0})+(f(\max m_{0})-H(0))(T-t),
\]
\[
\min m_{0}\leq m(x,T)\leq\max m_{0.}
\]
\end{rem}

\subsection{Estimates for the space-time gradient\label{subsec:Estimates grad}}

Given the operator $Q$ from (\ref{eq:ellip}), we recall that its
linearization at $u\in C^{2}(\overline{Q_{T}})$ is the linear, uniformly
elliptic operator
\begin{equation}
L_{u}(v)=-\tr(A(x,Du)D^{2}v)-D_{q}\tr(A(x,Du)D^{2}u)\cdot Dv+D_{q}b(x,Du)\cdot Dv.\label{eq:Linearization definition}
\end{equation}
The gradient estimate will be obtained through Bernstein's method.
Specifically, we will bound $||Du||_{C^{0}(\QT)}$ by evaluating the
linearization $L_{u}(v)$ at appropriately chosen functions $v(x,t)=\Phi(x,t,u,Du)$,
where $\Phi(x,t,z,q)$ is convex in $q$, exploiting the fact that,
roughly speaking, convex functions of the gradient are expected to
be subsolutions. For this purpose, we first obtain an explicit form
for the terms in (\ref{eq:Linearization definition}), as well as
a general expression for the linearization applied to such functions
$v$.
\begin{lem}
\label{lem:(Bernstein's-method)}Let $\Phi(x,t,p,s)\in C^{2}(\overline{Q_{T}}\times\mathbb{R}^{d+1})$,
assume that $u\in C^{3}(\overline{Q_{T}})$ solves $(\ref{eq:ellip}),$
and set $v(x,t)=\Phi(x,t,Du(x,t))$. Then, for each $q=(p,s)\in\mathbb{R}^{d+1}$,
and for each $\overline{x}=(x,t)\in Q_{T}$,
\begin{multline}
-D_{q}\textup{Tr}(AD^{2}u)\cdot q=2(-D_{p}HD_{xx}^{2}u+D_{x}u_{t})D_{pp}^{2}H\cdot p-\chi D_{p}(\text{\emph{Tr}}(D_{pp}^{2}HD_{xx}^{2}u))\cdot p\\
+\chi_{w}\text{\emph{Tr}}(D_{pp}^{2}HD_{xx}^{2}u)(s-D_{p}H\cdot p),\label{eq:D_q(A)}
\end{multline}
\begin{multline}
D_{q}b(x,Du)\cdot q=-(D_{p}HD_{xp}^{2}H)\cdot p-(D_{x}HD_{pp}^{2}H)\cdot p+(D_{x}fD_{pp}^{2}H)\cdot p\\
+\frac{1}{f_{m}}D_{x}f_{m}\cdot D_{p}H(-s+D_{p}H\cdot p)-\chi D_{p}\text{\emph{Tr}}(D_{xp}^{2}H)\cdot p+\chi_{w}\text{\emph{Tr}}(D_{xp}^{2}H)(s-D_{p}H\cdot p),\label{eq:D_q(b)}
\end{multline}
\begin{multline}
L_{u}v=-\text{\emph{\ensuremath{\tr}}}(D_{qq}^{2}\Phi D^{2}uAD^{2}u)-\emph{\ensuremath{\tr}}(AD_{\xo\xo}^{2}\Phi)-2\emph{\ensuremath{\tr}}(AD^{2}uD_{\xo q}^{2}\Phi)-D_{p}\Phi\cdot D_{x}b+D_{\xo}\Phi\cdot D_{q}b\\
+\sum_{i=1}^{d}\emph{\ensuremath{\tr}}(A_{x_{i}}D^{2}u)\Phi_{p_{i}}-D_{\xo}\Phi\cdot D_{q}\emph{\ensuremath{\tr}}(AD^{2}u).\label{eq:bernst}
\end{multline}
\end{lem}

\begin{proof}
Using (\ref{eq:matrix}),
\begin{align*}
-(D_{q}\tr(AD^{2}u))\cdot q= & -D_{q}(\tr((D_{p}H\otimes D_{p}H+\chi D_{pp}^{2}H)D_{xx}^{2}u)-2D_{p}H\cdot D_{x}u_{t})\cdot q\\
= & -2(D_{p}HD_{xx}^{2}uD_{pp}^{2}H)\cdot p-\chi D_{p}\tr(D_{pp}^{2}HD_{xx}^{2}u)\cdot p\\
 & -\chi_{w}\tr(D_{pp}^{2}HD_{xx}^{2}u)(-s+D_{p}H\cdot p)+2(D_{pp}^{2}HD_{x}u_{t})\cdot p\\
= & 2(-D_{p}HD_{xx}^{2}u+D_{x}u_{t})D_{pp}^{2}H\cdot p-\chi D_{p}(\tr(D_{pp}^{2}HD_{xx}^{2}u))\cdot p\\
 & +\chi_{w}\tr(D_{pp}^{2}HD_{xx}^{2}u)(s-D_{p}H\cdot p),
\end{align*}
which shows (\ref{eq:D_q(A)}). Equation (\ref{eq:D_q(b)}) is an
immediate consequence of (\ref{eq:first order term}). From the definition
of $v$, it follows that
\begin{align*}
Dv= & D_{\bar{x}}\Phi+D_{q}\Phi D^{2}u\;\;\text{ and \;\;}D^{2}v=D_{\xo\xo}^{2}\Phi+(D^{2}uD_{\xo q}^{2}\Phi+D_{q\xo}^{2}\Phi D^{2}u)+D^{2}uD_{qq}^{2}\Phi D^{2}u+D_{q}\Phi D^{3}u.
\end{align*}
Thus, differentiating the equation $Qu=0$ and taking the inner product
with $D_{q}\Phi$ yields
\begin{align*}
0= & \;D_{q}\Phi\cdot D_{\xo}(-\tr(A(x,Du(x,t))D^{2}u(x,t))+b(x,Du(x,t)))\\
= & -\tr(AD_{q}\Phi D^{3}u)-\sum_{i=1}^{d+1}\tr(A_{\xo_{i}}D^{2}u)\Phi_{q_{i}}-D_{q}\Phi D^{2}u\cdot D_{q}\tr(AD^{2}u)+D_{q}\Phi D^{2}u\cdot D_{q}b+D_{q}\Phi\cdot D_{\xo}b\\
= & -\tr(A(D^{2}v-(D_{\xo\xo}^{2}\Phi+(D^{2}uD_{\xo q}^{2}\Phi+D_{q\xo}^{2}\Phi)D^{2}u+D^{2}uD_{qq}^{2}\Phi D^{2}u))-D_{q}\tr(AD^{2}u)\cdot(Dv-D_{\xo}\Phi)\\
 & +D_{q}b\cdot(Dv-D_{\xo}\Phi)-\sum_{i=1}^{d+1}\tr(A_{\xo_{i}}D^{2}u)\Phi_{q_{i}}+D_{q}\Phi D_{\xo}b.
\end{align*}
Using the fact that $A$ and $b$ are independent of $t$, as well
as (\ref{eq:Linearization definition}), we obtain
\begin{align*}
0= & \;L_{u}v+\tr(A(D_{\xo\xo}^{2}\Phi+(D^{2}uD_{\xo q}^{2}\Phi+D_{q\xo}^{2}\Phi D^{2}u)+D^{2}uD_{qq}^{2}\Phi D^{2}u))+D_{q}\tr(AD^{2}u)\cdot D_{\xo}\Phi\\
 & -D_{q}b\cdot D_{\xo}\Phi-\sum_{i=1}^{d}\tr(A_{x_{i}}D^{2}u)\Phi_{p_{i}}+D_{p}\Phi\cdot D_{x}b\\
= & \;L_{u}v+\tr(D_{qq}^{2}\Phi D^{2}uAD^{2}u)+\tr(AD_{\xo\xo}^{2}\Phi)+2\tr(AD^{2}uD_{\xo q}^{2}\Phi)+D_{p}\Phi\cdot D_{x}b\\
 & -D_{q}b\cdot D_{\xo}\Phi-\sum_{i=1}^{d}\tr(A_{x_{i}}D^{2}u)\Phi_{p_{i}}+D_{\xo}\Phi\cdot D_{q}\tr(AD^{2}u),
\end{align*}
which proves (\ref{eq:bernst}).
\end{proof}
\begin{cor}
\label{cor:utbound}Let $(u,m)\in C^{3}(\overline{Q_{T}})\times C^{2}(\overline{Q_{T}})$
be a solution to (\ref{eq:mfg}), and set

\[
\eta_{0}=\min(\min_{\mathbb{T}^{d}}m_{0},\min_{\mathbb{T}^{d}}m(\cdot,T)),\;\eta_{1}=\max(\max_{\mathbb{T}^{d}}m_{0},\max_{\mathbb{T}^{d}}m(\cdot,T)).
\]
Then
\begin{equation}
L_{u}(u_{t})=0,\;\text{ and }\;-C_{0}-f_{1}(\eta_{1})\leq u_{t}\leq||H(\cdot,D_{x}u)||_{C^{0}(\overline{Q_{T}})}-f_{0}(\eta_{0}).\label{eq:u_tbound}
\end{equation}
\end{cor}

\begin{proof}
Letting $\Phi(x,t,p,s)=s$ in Lemma \ref{lem:(Bernstein's-method)},
since $D_{\xo}\Phi$, $D_{p}\Phi$, $D^{2}\Phi\equiv0$, it follows
that
\[
L_{u}(u_{t})=L_{u}(\Phi(x,t,Du))=0.
\]
Hence, the maximum and minimum values of $u_{t}$ are attained in
$\partial Q_{T}$, and (\ref{eq:u_tbound}) then follows immediately
from (\ref{eq: H>=00003D|p|^2-C}) and the HJ equation in (\ref{eq:mfg}).
\end{proof}
By Corollary \ref{cor:apriorim}, this result reduces the problem
to estimating $||D_{x}u||_{C^{0}},$ but it is also a key ingredient
for obtaining that bound, particularly due to the fact that the term
$||H(\cdot,D_{x}u)||_{C^{0}(\QT)}$ has coefficient $1$ in (\ref{eq:u_tbound}).
We now begin to simplify the quantity (\ref{eq:bernst}) for the specific
$\Phi$ that will be used in the proof of the gradient estimate, bounding
one of the dominant signed terms by a simpler expression, using matrix
algebra.

\begin{lem}
\label{lem:trace lemma} Assume that (\ref{eq:SE}) holds. For each
$(x,t,p,s)\in\overline{Q_{T}}\times\mathbb{R}^{d+1},$ set $\widetilde{H}(x,t,p,s)=H(x,p)$,
and define the matrix $\tilde{I}=(\delta^{ij}(1-\delta^{i,d+1}))_{i,j=1}^{d+1}$.
Then, for every $u\in C^{2}(Q_{T})$,
\begin{multline}
\emph{\text{Tr}}(D_{qq}^{2}\widetilde{H}(x,t,Du)D^{2}uA(x,Du)D^{2}u)\geq\frac{3}{4C_{0}}|-D_{x}u_{t}+D_{p}H(x,D_{x}u)D_{xx}^{2}u|^{2}\\
+\frac{1}{4C_{0}}\emph{\ensuremath{\tr}}(\tilde{I}D^{2}uAD^{2}u)+\frac{3\chi}{4C_{0}^{2}}|D_{xx}^{2}u|^{2}.\label{eq:trace ineq}
\end{multline}
\end{lem}

\begin{proof}
By (\ref{eq:strconv}),
\[
D_{qq}^{2}\widetilde{H}\geq\frac{1}{C_{0}}\tilde{I};
\]
thus, since the matrix $D^{2}uAD^{2}u$ is non-negative, multiplying
both sides of the inequality by this matrix and taking the trace of
both sides gives 
\begin{equation}
\tr(D_{pp}^{2}\widetilde{H}D^{2}uAD^{2}u)\geq\frac{1}{C_{0}}\tr(\tilde{I}D^{2}uAD^{2}u).\label{eq:DHtilde trace intermediate}
\end{equation}
Now, by (\ref{eq:matrix}) and (\ref{eq:strconv}),
\begin{multline}
\tr(\tilde{I}D^{2}uAD^{2}u)=\sum_{k=1}^{d}Du_{x_{k}}A\cdot Du_{x_{k}}=\sum_{k=1}^{d}|(D_{p}H,-1)\cdot Du_{x_{k}}|^{2}+\chi D_{x}u_{x_{k}}D_{pp}^{2}H\cdot D_{x}u_{x_{k}}\\
\geq\sum_{k=1}^{d}|D_{p}H\cdot D_{x}u_{x_{k}}-u_{tx_{k}}|^{2}+\frac{\chi}{C_{0}}|D_{x}u_{x_{k}}|^{2}=|D_{p}HD_{xx}^{2}u-D_{x}u_{t}|^{2}+\frac{\chi}{C_{0}}|D_{xx}^{2}u|^{2}.\label{eq:I_d trace intermediate}
\end{multline}
Combining (\ref{eq:DHtilde trace intermediate}) and (\ref{eq:I_d trace intermediate})
yields, as desired,
\begin{align*}
\tr(D_{pp}^{2}\widetilde{H}D^{2}uAD^{2}u) & =\frac{3}{4}\tr(D_{pp}^{2}\widetilde{H}D^{2}uAD^{2}u)+\frac{1}{4}\tr(D_{pp}^{2}\widetilde{H}D^{2}uAD^{2}u)\\
 & \geq\frac{3}{4C_{0}}|-D_{x}u_{t}+D_{p}HD_{xx}^{2}u|^{2}+\frac{3\chi}{4C_{0}^{2}}|D_{xx}^{2}u|^{2}+\frac{1}{4C_{0}}\tr(\tilde{I}D^{2}uAD^{2}u).
\end{align*}
\end{proof}
The next Lemma continues to simplify the linearizations. Since one
of the dominant signed terms will later be shown to be of order $|D_{x}u|^{4}$,
the goal will be to bound everything else by $(4-\epsilon)^{\text{}}$th
powers of $|D_{x}u|$, $(2-\epsilon)^{\text{}}$th powers of $u_{t}$
(dealing with these through Corollary \ref{cor:utbound}), and second
derivative terms that can be dealt with using the other dominant term
(\ref{eq:trace ineq}). The usage of $\Phi(x,t,D_{x}u,u_{t})=H(x,D_{x}u)$,
as opposed to a more standard choice such as $|D_{x}u|^{2}$ or $|Du|^{2}$,
is crucial in the next two results, in order to produce structural
cancellation of terms that can not be otherwise estimated, as well
as to be able to use (\ref{eq:u_tbound}) without gaining any constant
factors in the process. 
\begin{lem}
\label{lem: L(v1) and L(v2)} Assume that (\ref{eq:SE}) holds. Let
$u\in C^{3}(\overline{Q_{T}})$ be a solution to (\ref{eq:ellip}),
and let $c,c'\in\mathbb{R}.$ Define
\[
\widetilde{u}=u+c(T-t)+c',\;v_{1}=\frac{\utilde^{2}}{2},\;v_{2}=H(\cdot,D_{x}u).
\]
Then, for each $(x,t)\in Q_{T}$, there exists $C(x,t)>0$, with \textup{
\[
C(x,t)=C\left(C_{0},||\utilde||_{C^{0}(Q_{T})},\frac{1}{\chi(x,f(x,m(x,t)))},|h_{w}(x,f(x,m(x,t)))|,c\right),
\]
such that}
\begin{multline}
L_{u}(v_{1})\leq-\frac{1}{2}|-u_{t}+D_{p}H(x,D_{x}u)D_{x}u|^{2}-\frac{1}{C_{0}}\chi|D_{x}u|^{2}+C(x,t)(1+|D_{x}u|^{2+\tau}+\chi|D_{xx}^{2}u|^{2}\\
|f|^{\tau}|D_{x}u|^{2}+|f|(1+|D_{x}u|^{\tau})+|f|^{(1+\tau)/2}|D_{x}u|+\chi+|-D_{x}u_{t}+D_{xx}^{2}uD_{p}H|^{2}).\label{eq:Lutilde}
\end{multline}
and
\begin{multline}
L_{u}(v_{2})\leq\frac{-1}{2C_{0}}|-D_{x}u_{t}+D_{p}HD_{xx}^{2}u|^{2}-\frac{\chi}{2C_{0}^{2}}|D_{xx}^{2}u|^{2}+C(x,t)(1+|D_{x}u|^{3+\tau}+\chi(1+|D_{x}u|^{1+\tau})\\
+\chi^{(1+\tau)/2}|D_{x}u|^{2}+|f|(1+|D_{x}u|^{1+\tau})+|D_{x}u|^{2}|f|^{(1+\tau)/2}).\label{eq:L(H)}
\end{multline}
\end{lem}

\begin{proof}
Throughout this proof, the number $C=C(x,t)$ may increase at each
step, with its size depending on $(x,t)$ only monotonically through
$\frac{1}{\chi}$ and $|h_{w}|$. For this reason, there is no loss
of generality in assuming
\begin{equation}
\frac{1}{\chi}+|h_{w}|\leq C.\label{eq:C(x,t) WLOG}
\end{equation}
Observe first that, since $D^{2}\utilde=D^{2}u$, one has $-\tr(A(x,Du)D^{2}\utilde)=-b(x,Du)$.
Therefore,
\begin{align*}
-\tr(A(x,Du)D^{2}v_{1}) & =-\utilde b(x,Du)-D\utilde A\cdot D\utilde,\\
-D_{q}\tr(A(x,Du)D^{2}u)\cdot Dv_{1}+D_{q}b(x,Du)\cdot Dv_{1} & =-D_{q}\tr(AD^{2}u)\cdot\utilde D\utilde+D_{q}b\cdot\utilde D\utilde.
\end{align*}
Consequently, by (\ref{eq:matrix}) and (\ref{eq:qg2}),
\begin{align}
L_{u}(v_{1})= & -\utilde b-D\utilde A\cdot D\utilde+\utilde(-D_{q}\tr(AD^{2}u)\cdot D\utilde+D_{q}b\cdot D\utilde)\label{eq:linearization utilde}\\
= & -\utilde b-D\utilde(D_{p}H,-1)\otimes(D_{p}H,-1)\cdot D\utilde-\chi D_{x}\utilde D_{pp}^{2}H\cdot D_{x}\utilde\nonumber \\
 & +\utilde(-D_{q}\tr(AD^{2}u)\cdot D\utilde+D_{q}b\cdot D\utilde)\nonumber \\
\leq & -|-\utilde_{t}+D_{p}H\cdot D_{x}u|^{2}-\frac{1}{C_{0}}\chi|D_{x}\utilde|^{2}+\utilde(-b-D_{q}\tr(ADu)\cdot D\utilde+D_{q}b\cdot D\utilde)\nonumber \\
= & -|-\utilde_{t}+D_{p}H\cdot D_{x}u|^{2}-\frac{1}{C_{0}}\chi|D_{x}\utilde|^{2}+\utilde(J_{1}+J_{2}+J_{3}).\nonumber 
\end{align}
The next task will be to estimate the terms $J_{i}$. In view of (\ref{eq:first order term}),
(\ref{eq: H>=00003D|p|^2-C}), (\ref{eq: Hx, Hxx bound}), (\ref{eq: Hxp bound}),
(\ref{growth chi}), and (\ref{eq:fx, fxx bound}),
\begin{multline}
|J_{1}|=|b(x,Du)|=|-D_{x}H\cdot D_{p}H+D_{x}f\cdot D_{p}H-\chi\tr(D_{xp}^{2}H)|\\
\leq C(1+|D_{x}u|^{2+\tau}+|f|^{(1+\tau)/2}(1+|D_{x}u|)+|f|(1+|D_{x}u|^{\tau})).\label{eq: bound J1}
\end{multline}
As for $J_{2}$, (\ref{eq: H>=00003D|p|^2-C}), (\ref{eq:qg2}), (\ref{eq:D_q(A)}),
and (\ref{eq:C(x,t) WLOG}) imply that
\begin{multline}
|J_{2}|=|-D_{q}\tr(AD^{2}u)\cdot D\utilde|=|2(-D_{p}HD_{xx}^{2}u+D_{x}u_{t})D_{pp}^{2}H\cdot D_{x}u-\chi D_{p}(\tr(D_{pp}^{2}HD_{xx}^{2}u))\cdot D_{x}u\\
+2hh_{w}\tr(D_{pp}^{2}HD_{xx}^{2}u)(\utilde_{t}-D_{p}H\cdot D_{x}u)|\leq C(1+|-D_{x}u_{t}+D_{p}HD_{xx}^{2}u|^{2}+|D_{x}u|^{2}+\chi\\
+\chi|D_{xx}^{2}u|^{2})+\frac{1}{8(||\utilde||+1)}|-\utilde_{t}+D_{p}H\cdot D_{x}u|^{2}.\label{eq: bound J2}
\end{multline}
By assumptions (\ref{eq:Hxxp, Hxpp bound}), (\ref{growth chi}),
(\ref{eq: fmx bound}), and (\ref{eq:fx, fxx bound}), together with
(\ref{eq: Hx, Hxx bound}), (\ref{eq:D_q(b)}) and (\ref{eq:C(x,t) WLOG}),
\begin{multline}
|J_{3}|=|D_{q}b(x,Du)\cdot D\utilde|=|-(D_{p}HD_{xp}^{2}H)\cdot D_{x}u-(D_{x}HD_{pp}^{2}H)\cdot D_{x}u\\
+(D_{x}fD_{pp}^{2}H)\cdot D_{x}u+\left(\frac{1}{f_{m}}D_{x}f_{m}\cdot D_{p}H\right)(-\utilde_{t}+D_{p}H\cdot D_{x}u)-\chi D_{p}(\tr(D_{xp}^{2}H))\cdot D_{x}u\\
+\chi_{w}\tr(D_{xp}^{2}H)(\utilde_{t}-D_{p}H\cdot D_{x}u)|\leq C(1+|D_{x}u|^{2+\tau}+(1+|f|^{(1+\tau)/2})|D_{x}u|\\
+(1+\frac{1}{\chi})(1+|f|^{\tau/2})(1+|D_{x}u|)|-\utilde_{t}+D_{p}H\cdot D_{x}u|+(1+|f|)(1+|D_{x}u|^{\tau})\\
+(1+|D_{x}u|^{\tau})|-\utilde_{t}+D_{p}H\cdot D_{x}u|)\leq C(1+|D_{x}u|^{2+\tau}+|f|^{(1+\tau)/2}|D_{x}u|\\
+|f|^{\tau}|D_{x}u|^{2}+|f||D_{x}u|^{\tau})+\frac{1}{8(||\utilde||+1)}|-\utilde_{t}+D_{p}H\cdot D_{x}u|^{2}.\label{eq: bound J3}
\end{multline}
Finally, using (\ref{eq: bound J1}), (\ref{eq: bound J2}), and (\ref{eq: bound J3})
in (\ref{eq:linearization utilde}) yields
\begin{align*}
L_{u}(v_{1})\leq-\frac{3}{4} & |-\utilde_{t}+D_{p}HD_{x}u|^{2}-\frac{1}{C_{0}}\chi|D_{x}u|^{2}+C(1+|D_{x}u|^{2+\tau}+\chi|D_{xx}^{2}u|^{2}+|f|^{\tau}|D_{x}u|^{2}\\
 & +|f|(1+|D_{x}u|^{\tau})+|f|^{(1+\tau)/2}|D_{x}u|+\chi+|-D_{x}u_{t}+D_{xx}^{2}uD_{p}H|^{2}),
\end{align*}
which proves (\ref{eq:Lutilde}). 

Next is the proof of (\ref{eq:L(H)}). In view of Lemma \ref{lem:(Bernstein's-method)},
recalling that $\tilde{H}(x,t,Du):=H(x,Du)$,
\begin{align*}
L_{u}v_{2}= & L_{u}(\tilde{H}(x,t,Du))\\
= & -\tr(D_{qq}^{2}\tilde{H}D^{2}uAD^{2}u)-\tr(AD_{\xo\xo}^{2}\tilde{H})-2\tr(AD^{2}uD_{\xo q}^{2}\tilde{H})-D_{p}\tilde{H}\cdot D_{x}b\\
 & +D_{\xo}\tilde{H}\cdot D_{q}b+\sum_{i=1}^{d}\tr(A_{x_{i}}D^{2}u)\tilde{H}_{p_{i}}-D_{\xo}\tilde{H}\cdot D_{q}\tr(AD^{2}u),
\end{align*}
and Lemma \ref{lem:trace lemma} then implies
\begin{align}
L_{u}v_{2}\leq & \frac{-3}{4C_{0}}|-D_{x}u_{t}+D_{p}HD_{xx}^{2}u|^{2}-\frac{3\chi}{4C_{0}^{2}}|D_{xx}^{2}u|^{2}-\frac{1}{4C_{0}}\tr(\tilde{I}D^{2}uAD^{2}u)\label{eq:Lv2 intermediate}\\
 & -\tr(AD_{\xo\xo}^{2}\tilde{H})-2\tr(AD^{2}uD_{\xo q}^{2}\tilde{H})-D_{p}H\cdot D_{x}b+D_{x}H\cdot D_{p}b\nonumber \\
 & +\sum_{i=1}^{d}\tr(A_{x_{i}}D^{2}u)H_{p_{i}}-D_{x}H\cdot D_{p}\tr(AD^{2}u).\nonumber \\
= & \frac{-3}{4C_{0}}|-D_{x}u_{t}+D_{p}HD_{xx}^{2}u|^{2}-\frac{3\chi}{4C_{0}^{2}}|D_{xx}^{2}u|^{2}-\frac{1}{4C_{0}}\tr(\tilde{I}D^{2}uAD^{2}u)\nonumber \\
 & +K_{1}+K_{2}+K_{3}+K_{4}+K_{5}+K_{6}.\nonumber 
\end{align}
As before, we proceed to estimate the $K_{i}$. Starting with $K_{1}$,
we observe that, by (\ref{eq:matrix}), (\ref{eq:strconv}), (\ref{eq: H>=00003D|p|^2-C}),
and (\ref{eq: Hxp bound}),
\begin{align}
|K_{1}|= & |\tr(AD_{\xo\xo}^{2}\tilde{H})|\leq C(1+|D_{x}u|^{1+\tau})|A|\leq C(1+|D_{x}u|^{1+\tau})(1+|D_{x}u|^{2}+\chi).\label{eq:K1 bound}
\end{align}
Similarly,  using the Cauchy--Schwarz inequality, 
\begin{multline}
|K_{2}|=|2\tr(AD^{2}u(\tilde{I}D_{\xo q}^{2}\tilde{H}))|=2|\tr(D_{\xo q}^{2}\tilde{H}A(\tilde{I}D^{2}u)^{T})|\leq\frac{1}{4C_{0}}\tr((\tilde{I}D^{2}u)A(\tilde{I}D^{2}u)^{T})\\
+C\tr(D_{\xo q}^{2}\tilde{H}A(D_{\xo q}^{2}\tilde{H})^{T})\leq\frac{1}{4C_{0}}\tr(\tilde{I}D^{2}uAD^{2}u)+C|D_{xp}^{2}H|^{2}(1+|D_{x}u|^{2}+\chi)\\
\leq\frac{1}{4C_{0}}\tr(\tilde{I}D^{2}uAD^{2}u)+C(1+|D_{x}u|^{2(1+\tau)}+\chi(1+|D_{x}u|^{2\tau})).\label{eq: K2 bound}
\end{multline}
Next, we will estimate $|K_{3}+K_{4}|$. Differentiating the equation
(\ref{eq:first order term}) with respect to $x$, we obtain
\begin{multline}
D_{x}b(x,p,s)=-D_{xx}^{2}H\cdot D_{p}H-D_{x}HD_{xp}^{2}H+D_{p}HD_{xx}^{2}f+\frac{1}{f_{m}}D_{p}H(D_{x}f_{m}\otimes(-D_{x}f+D_{x}H))\\
+D_{x}fD_{xp}^{2}H+\left(D_{x}f-mD_{x}f_{m}+\frac{mf_{mm}}{f_{m}}D_{x}f\right)\tr(D_{xp}^{2}H)-\chi_{w}D_{x}H\tr(D_{xp}^{2}H)-\chi D_{x}\tr(D_{xp}^{2}H).\label{eq:Lv2 inter 1}
\end{multline}
Consequently, (\ref{eq: H>=00003D|p|^2-C}), (\ref{eq: Hx, Hxx bound}),
(\ref{eq: Hxp bound}), (\ref{growth chi}), (\ref{eq: fmx bound}),
(\ref{eq:fx, fxx bound}), and (\ref{eq:C(x,t) WLOG}) yield
\begin{align*}
|-D_{x}b(x,p,s)+\frac{1}{f_{m}}D_{p}HD_{x}f_{m}\otimes D_{x}H|\leq & C(1+|p|^{2+\tau}+(1+|p|)|f|^{(1+\tau)/2}+|f|(1+|p|^{\tau})),
\end{align*}
so that, setting $z_{1}=\frac{1}{f_{m}}(D_{p}H\cdot D_{x}f_{m})(D_{x}H\cdot D_{p}H),$
\begin{multline}
|K_{3}+z_{1}|=|-D_{x}b(x,Du)\cdot D_{p}H+z_{1}|\leq C(1+|D_{x}u|^{2+\tau}+(1+|D_{x}u|)|f|^{(1+\tau)/2}\\
+|f|(1+|D_{x}u|^{\tau}))(1+|D_{x}u|)\leq C(1+|D_{x}u|^{3+\tau}+(1+|D_{x}u|^{2})|f|^{(1+\tau)/2}+|f|(1+|D_{x}u|^{1+\tau})).\label{eq: K3 bound}
\end{multline}
On the other hand, by (\ref{eq:D_q(b)}),
\begin{multline*}
|K_{4}-z_{1}|=|D_{p}b\cdot D_{x}H-z_{1}|=|-(D_{p}HD_{xp}^{2}H)\cdot D_{x}H-(D_{x}HD_{pp}^{2}H)\cdot D_{x}H\\
+(D_{x}fD_{pp}^{2}H)\cdot D_{x}H+z_{1}-\chi D_{p}\tr(D_{xp}^{2}H)\cdot D_{x}H+\chi_{w}\tr(D_{xp}^{2}H)(D_{p}H\cdot D_{x}H)-z_{1}|.
\end{multline*}
The terms $z_{1}$ and $-z_{1}$ then cancel out, and therefore (\ref{eq:Hxxp, Hxpp bound}),
(\ref{eq: Hx, Hxx bound}), (\ref{eq: Hxp bound}), and (\ref{growth chi})
yield
\begin{align}
|K_{4}-z_{1}|\leq & C(1+|D_{x}u|^{2+2\tau}+|f|^{(1+\tau)/2}(1+|D_{x}u|^{1+\tau})+|f||D_{x}u|^{2\tau}).\label{eq:K4 bound}
\end{align}
The inequalities (\ref{eq: K3 bound}) and (\ref{eq:K4 bound}) thus
imply
\begin{equation}
|K_{3}+K_{4}|\leq C(1+|D_{x}u|^{3+\tau}+(1+|D_{x}u|^{2})|f|^{(1+\tau)/2}+|f|(1+|D_{x}u|^{1+\tau})).\label{eq:K3+K4 bound}
\end{equation}
The terms $K_{5}$ and $K_{6}$ will also be treated jointly. Let
$(x,p,s)\in\mathbb{T}^{d}\times\mathbb{R}^{d+1}$, and set $w=-s+H(x,p)$.
It follows from (\ref{eq:matrix}), (\ref{eq: fmx bound}), (\ref{eq:Hxxp, Hxpp bound}),
and (\ref{growth chi}) that
\begin{multline*}
A_{x_{i}}(x,p,s)=(D_{p}H_{x_{i}},0)\otimes(D_{p}H,-1)+(D_{p}H,-1)\otimes(D_{p}H_{x_{i}},0)+\chi_{w}H_{x_{i}}D_{qq}^{2}\tilde{H}\\
+O(1+|w|^{\tau/2}+\chi{}^{(1+\tau)/2}+\chi(1+|p|)^{-1+\tau})\tilde{I}.
\end{multline*}
Therefore, using (\ref{eq:D_q(A)}), and writing $z_{2}=$$\chi_{w}\tr(D_{pp}^{2}HD_{xx}^{2}u)(D_{p}H\cdot D_{x}H)$,\vspace*{-0.24cm}
\begin{multline*}
|K_{5}+K_{6}|=|\sum_{i=1}^{d}\tr(A_{x_{i}}D^{2}u)H_{p_{i}}-D_{p}\tr(AD^{2}u)\cdot D_{x}H|\leq|2(D_{p}HD_{xx}^{2}u-D_{x}u_{t})D_{px}^{2}H\cdot D_{p}H+z_{2}\\
+2(-D_{p}HD_{xx}^{2}u+D_{x}u_{t})D_{pp}^{2}H\cdot D_{x}H-\chi D_{p}(\tr(D_{pp}^{2}HD_{xx}^{2}u))\cdot D_{x}H-z_{2}|+C(1+|f|^{\tau/2}\\
+\chi^{(1+\tau)/2}+\chi(1+|D_{x}u|)^{-1+\tau})|D_{xx}^{2}u|(1+|D_{x}u|).
\end{multline*}
 Once more, cancellation occurs and, consequently, (\ref{eq:C(x,t) WLOG}),
(\ref{eq:qg2}), (\ref{eq: H>=00003D|p|^2-C}), (\ref{eq: Hx, Hxx bound}),
and (\ref{eq: Hxp bound}) imply that 
\begin{multline}
|K_{5}+K_{6}|\leq\frac{1}{4C_{0}}|D_{p}HD_{xx}^{2}u-D_{x}u_{t}|^{2}+\frac{\chi}{4C_{0}^{2}}|D_{xx}^{2}u|^{2}+C(1+|D_{x}u|^{2+2\tau}\\
+(1+|D_{x}u|^{2})(1+|f|^{\tau}+\chi^{(1+\tau)/2})+\chi(1+|D_{x}u|^{2\tau})).\label{eq:K5+K6 bound}
\end{multline}
Using (\ref{eq:K1 bound}), (\ref{eq: K2 bound}), (\ref{eq:K3+K4 bound}),
and (\ref{eq:K5+K6 bound}) in (\ref{eq:Lv2 intermediate}) yields
(\ref{eq:L(H)}), completing the proof.
\end{proof}
We can now obtain the a priori gradient bound in terms of bounds for
the solution $u$ and the terminal density $m(\cdot,T)$, which were
obtained in the previous subsection. 
\begin{lem}
\label{lem: gradient a priori bound} Assume that (\ref{eq:SE}) holds,
and let $(u,m)\in C^{3}(\overline{Q_{T}})\times C^{2}(\overline{Q_{T}})$
be a solution to (\ref{eq:mfg}). For $K>0$, set 
\begin{align}
\beta_{K}= & ||f||_{C^{1}(\mathbb{T}^{d}\times[\frac{1}{K},K])}+||D_{x}g||_{C^{1}(\mathbb{T}^{d}\times[\frac{1}{K},K])}+\left\Vert \frac{1}{\chi}\right\Vert _{C^{0}(\mathbb{T}^{d}\times[-K,\infty))}+||h_{w}||_{C^{0}(\mathbb{T}^{d}\times[-K,\infty))}.\label{eq:beta_K}
\end{align}
There exist constants $C,$ $C_{1}$ with
\[
C=C(C_{1},\beta_{C_{1}}),\;C_{1}=C_{1}\left(C_{0},T,\frac{1}{T},\frac{1}{1-\tau},||u||_{C^{0}(\QT)},\max m(T),\frac{1}{\min m(T)},f_{0}(\min_{\mathbb{T}^{d}}m(T))^{-}\right),
\]
such that
\[
||Du||_{C^{0}(\QT)}\leq C.
\]
\end{lem}

\begin{proof}
As was mentioned, the proof will proceed through Bernstein's method.
By Corollary \ref{cor:utbound}, it is sufficient to bound the space
gradient. Since the estimate will be up to the boundary, as in \cite{Lions},
we linearize the HJ equation that holds at the extremal times: 
\[
T_{u}v=-v_{t}+D_{p}H(x,Du)D_{x}v.
\]
We now normalize $u$ to have a prescribed sign at the initial and
terminal times. That is, we set
\[
\utilde=u+||u||_{C^{0}(\QT)}+1-\frac{2(||u||_{C^{0}(\QT)}+1)}{T}(T-t),
\]
so that 
\begin{equation}
|\utilde|\leq C,\quad\utilde(\cdot,0)\leq-1,\;\utilde(\cdot,T)\geq1,\label{signs utilde}
\end{equation}
and define
\[
v(x,t)=H(x,D_{x}u)+\frac{c_{1}}{2}\utilde^{2},
\]
where $0<c_{1}\leq1$ is a constant to be chosen later. Let $(x_{0},t_{0})\in\overline{Q_{T}}$
be a point where $v$ achieves its maximum value. We will distinguish
three cases:

\textbf{Case 1.} $t_{0}=T$. Then $D_{x}v=0$, $v_{t}\geq0$. Therefore,
(\ref{signs utilde}), (\ref{eq:strconv}) (\ref{qg1}), (\ref{eq: Hx, Hxx bound}),
and the HJ equation in (\ref{eq:mfg}), together with the fact that
$m(\cdot,T)=g^{-1}(\cdot,u(\cdot,T))$, yield
\begin{align*}
0\geq T_{u}v= & T_{u}(H(x_{0},D_{x}u))+c_{1}\utilde(-\utilde_{t}+D_{p}H(x_{0},D_{x}u)\cdot D_{x}\utilde)\\
= & D_{x}(f(x_{0},m(x_{0},T)))\cdot D_{p}H(x_{0},D_{x}u)+c_{1}\utilde(-u_{t}+D_{p}H(x_{0},D_{x}u)\cdot D_{x}u-C)\\
\geq & \left(D_{x}f-\frac{f_{m}}{g_{m}}D_{x}g+\frac{f_{m}}{g_{m}}D_{x}u\right)\cdot D_{p}H+c_{1}\utilde(-u_{t}+2H-C)\\
\geq & -C\left(1+\frac{f_{m}}{g_{m}}\right)(1+|D_{x}u|)+\frac{f_{m}}{g_{m}}D_{p}H\cdot D_{x}u+c_{1}\utilde(f+H-C)\\
\geq & -C\left(1+\frac{f_{m}}{g_{m}}\right)(1+|D_{x}u|)+2\left(c_{1}\utilde+\frac{f_{m}}{g_{m}}\right)H.
\end{align*}
Thus, by (\ref{eq: H>=00003D|p|^2-C}),
\[
|H(x_{0},Du(x_{0},t_{0}))|\leq C.
\]

\textbf{Case 2.} $t_{0}=0$. Similarly, we obtain $D_{x}v=0$, $v_{t}\leq0$,
and, since $\utilde(\cdot,0)\leq-1$,
\begin{multline*}
0\leq T_{u}v=D_{x}(f(x_{0},m_{0}(x_{0})))\cdot D_{p}H+c_{1}\utilde(-u_{t}+D_{p}H\cdot D_{x}u-C)\\
\leq C(1+|D_{x}u|)+c_{1}\utilde(f(x_{0},m_{0})+H-C)\leq C(1+|D_{x}u|)+C+c_{1}\utilde H.
\end{multline*}
This implies $-c_{1}\utilde(H(x_{0},D_{x}u))\leq C(1+|D_{x}u|),$
and so, we conclude once more that
\[
|H(x_{0},Du(x_{0},t_{0}))|\leq C.
\]

\textbf{Case 3.} $0<t_{0}<T$. Then $Dv=0$, $D^{2}v\leq0$, which
yields
\[
0\leq L_{u}v.
\]
In order to make use of Lemma \ref{lem: L(v1) and L(v2)}, it is necessary
to eliminate the $(x_{0},t_{0})$ dependence of the ``constant''~$C(x_{0},t_{0})$
from the Lemma, which amounts to establishing an a priori upper bound
on the quantities $1/\chi$ and $|h_{w}|$ at the point $(x_{0},t_{0})$.
By (\ref{eq:f polynomial growth}) and (\ref{growth chi}), $1/\chi$
and $|h_{w}|=|\chi_{w}/2\sqrt{\chi}|$ are both bounded above as $w\rightarrow\infty$,
so it is enough to establish a lower bound for $w=f(x_{0},m(x_{0},t_{0}))$.
By Corollary \ref{cor:utbound}, there exists a point $(x_{1},t_{1})\in\partial Q_{T}$
where $u_{t}$ achieves its maximum value. Then, since $(x_{0},t_{0})$
is a maximum point for $v$, and the initial and terminal densities
are both bounded below a priori,
\begin{multline*}
f(x_{0},m(x_{0},t_{0}))=-u_{t}(x_{0},t_{0})+H(x_{0},D_{x}u(x_{0},t_{0}))\geq-u_{t}(x_{1},t_{1})+H(x_{1},D_{x}u(x_{1},t_{1}))-\frac{c_{1}}{2}||\utilde||_{C^{0}(\QT)}^{2}\\
=f(x_{1},m(x_{1},t_{1}))-\frac{c_{1}}{2}||\utilde||_{C^{0}(\QT)}^{2}\geq f_{0}(m(x_{1},t_{1}))-C\geq-C.
\end{multline*}
This estimate, together with (\ref{qg1}) and (\ref{eq: H>=00003D|p|^2-C}),
allows us to identify the dominant power of $|D_{x}u|$ in the linearization,
\begin{align}
|-u_{t}+D_{p}H\cdot D_{x}u|^{2} & \geq(f+H)^{2}-C\geq\frac{1}{2C_{0}^{2}}|D_{x}u|^{4}-C.\label{eq: dominant term}
\end{align}
Now, because of the form of the estimate in Lemma \ref{lem: L(v1) and L(v2)},
it is also necessary to be able to compare powers of $|f|$ with powers
of $|D_{x}u|$. By Corollary \ref{cor:utbound} and (\ref{eq: H>=00003D|p|^2-C}),
\begin{align}
f(x_{0},m(x_{0},t_{0})) & \leq-u_{t}(x_{0},t_{0})+H(x_{0},D_{x}u(x_{0},t_{0}))\leq C+2H(x_{0},D_{x}u(x_{0},t_{0}))\leq C(1+|D_{x}u|^{2}).\label{eq: f upper bound}
\end{align}
With these preliminaries done, we now apply Lemma \ref{lem: L(v1) and L(v2)},
obtaining
\begin{multline}
0\leq L_{u}(v)=L_{u}\tilde{H}+c_{1}L_{u}(\frac{\utilde^{2}}{2})\leq\frac{-1}{2C_{0}}|-D_{x}u_{t}+D_{p}HD_{xx}^{2}u|^{2}-\frac{\chi}{2C_{0}^{2}}|D_{xx}^{2}u|^{2}+C(1+|D_{x}u|^{3+\tau}\\
+\chi(1+|D_{x}u|^{1+\tau})+\chi^{(1+\tau)/2}|D_{x}u|^{2}+|f|(1+|D_{x}u|^{1+\tau})+|D_{x}u|^{2}|f|^{(1+\tau)/2})\\
-\frac{c_{1}}{2}|-u_{t}+D_{p}H\cdot D_{x}u|^{2}-\frac{c_{1}}{C_{0}}\chi|D_{x}u|^{2}+Cc_{1}(\chi|D_{xx}^{2}u|^{2}+|-D_{x}u_{t}+D_{xx}^{2}uD_{p}H|^{2}).\label{eq:L_u(v) quote weak}
\end{multline}
Applying (\ref{eq: dominant term}) and (\ref{eq: f upper bound})
yields
\begin{multline*}
0\leq-\frac{c_{1}}{4C_{0}^{2}}|D_{x}u|^{4}-\frac{c_{1}}{C_{0}}\chi|D_{x}u|^{2}-\frac{1}{2C_{0}}|-D_{x}u_{t}+D_{p}HD_{xx}^{2}u|^{2}-\frac{\chi}{2C_{0}^{2}}|D_{xx}^{2}u|^{2}\\
+C(1+|D_{x}u|^{3+\tau}+\chi(1+|D_{x}u|^{1+\tau})+\chi^{(1+\tau)/2}|D_{x}u|^{2})+Cc_{1}(\chi|D_{xx}^{2}u|^{2}+|-D_{x}u_{t}+D_{xx}^{2}uD_{p}H|^{2})\\
\leq-\frac{c_{1}}{4C_{0}^{2}}|D_{x}u|^{4}-\frac{c_{1}}{C_{0}}\chi|D_{x}u|^{2}-\frac{1}{2C_{0}}|-D_{x}u_{t}+D_{p}HD_{xx}^{2}u|^{2}-\frac{\chi}{2C_{0}^{2}}|D_{xx}^{2}u|^{2}+C(1+|D_{x}u|^{3+\tau}\\
+\chi(1+|D_{x}u|^{1+\tau}))+(\frac{C}{c_{1}}+\frac{c_{1}}{2C_{0}}\chi)|D_{x}u|^{2}+Cc_{1}(\chi|D_{xx}^{2}u|^{2}+|-D_{x}u_{t}+D_{xx}^{2}uD_{p}H|^{2}).
\end{multline*}
Now, fix $c_{1}$ satisfying $c_{1}<\frac{1}{4C_{0}C(1+C_{0})}$,
where $C$ is as in the previous line. This gives
\begin{align*}
0\leq & -\frac{c_{1}}{4C_{0}^{2}}|D_{x}u|^{4}-\frac{1}{2C_{0}}\chi(c_{1}|D_{x}u|^{2}-2CC_{0}(1+|D_{x}u|^{1+\tau}))+C(1+|D_{x}u|^{3+\tau})+\frac{C}{c_{1}}|D_{x}u|^{2},
\end{align*}
which may be rearranged as
\[
\frac{c_{1}}{4C_{0}^{2}}|D_{x}u|^{4}+\frac{1}{2C_{0}}\chi(c_{1}|D_{x}u|^{2}-2CC_{0}|D_{x}u|^{1+\tau}-2CC_{0})\leq C(1+|D_{x}u|^{3+\tau})+\frac{C}{c_{1}}|D_{x}u|^{2}.
\]
This finally implies that
\[
c_{1}|D_{x}u|^{2}-2CC_{0}|D_{x}u|^{1+\tau}-2CC_{0}\leq0\;\text{ or \;}\frac{c_{1}}{4C_{0}^{2}}|D_{x}u|^{4}\leq C(1+|D_{x}u|^{3+\tau})+\frac{C}{c_{1}}|D_{x}u|^{2},
\]
either of which yields
\[
|H(x_{0},D_{x}u(x_{0},t_{0}))|\leq C.
\]
\end{proof}
We now summarize all of the a priori bounds obtained in this section.
\begin{thm}
\label{thm:Full C^1 a priori bound} Assume that (\ref{eq:SE}) holds,
let $(u,m)\in C^{3}(\overline{Q_{T}})\times C^{2}(\overline{Q_{T}})$
be a solution to (\ref{eq:mfg}), and let $\beta$ be defined by (\ref{eq:beta_K}).
Then there exist constants $L,L_{1},K,K_{1}$, with
\[
L=\left(L_{1},|g_{1}f_{0}^{-1}(L_{1})|,|g_{0}f_{1}^{-1}(-L_{1})|,g_{0}^{-1}g_{1}f_{0}^{-1}(L_{1}),\frac{1}{g_{1}^{-1}g_{0}f_{1}^{-1}(-L_{1}))}\right),\;\;\;L_{1}=L_{1}(C_{0},T),
\]
\[
K=(K_{1},\beta_{K_{1}}),\;\;\;K_{1}=K_{1}\left(L,\frac{1}{T},\frac{1}{1-\tau},f_{0}\left(\frac{1}{L}\right)^{-}\right),
\]
such that
\[
||u||_{C^{0}(\overline{Q_{T}})}+||m(T)||_{C^{0}(\mathbb{T}^{d})}+\left\Vert \frac{1}{m(T)}\right\Vert _{C^{0}(\mathbb{T}^{d})}\leq L\text{\;\; and }\;\;\;||Du||_{C^{0}(\overline{Q_{T}})}\leq K.
\]
\end{thm}

\begin{proof}
This result follows simply by the successive application of Lemma
\ref{lem:aprioriu}, Corollary \ref{cor:apriorim}, and Lemma \ref{lem: gradient a priori bound}.
\end{proof}
The following variation of Theorem \ref{thm:Full C^1 a priori bound}
shows that, in the standard case where $H(x,p)=H(p)-V(x)$ and $f(x,m)=f(m)$,
the condition (\ref{growth chi}) which requires $f$ to grow at most
polynomially may be significantly relaxed.
\begin{thm}
\label{thm:V(x) a priori C^1}The conclusion of Theorem \ref{thm:Full C^1 a priori bound}
still holds if condition (\ref{growth chi}) is replaced by:
\begin{align}
\tag{HFX*}D_{xp}^{2}H & ,D_{xm}^{2}f\equiv0,\text{ and }\limsup_{x\in\mathbb{T}^{d},w\rightarrow\infty}|h_{w}(x,w)|<\infty.\label{eq:separated potential}
\end{align}
\end{thm}

\begin{proof}
We simply address all of the instances in which condition (\ref{growth chi})
has been used so far. In the proofs of Lemma \ref{lem:aprioriu},
Corollary \ref{cor:apriorim}, and Lemma \ref{lem: L(v1) and L(v2)},
(\ref{growth chi}) was exclusively used to estimate either space
derivatives $D_{x}f,\;D_{x}H$, or terms that involve mixed derivatives
$D_{xm}^{2}f,\;$$D_{xp}^{2}H$. With (\ref{eq:separated potential})
in place, such terms are, respectively, either bounded in $C^{1}$
norm or trivially zero. Condition (\ref{growth chi}) was also used
in the proof of Lemma \ref{lem: gradient a priori bound} in order
to obtain a bound for $|h_{w}|$ as $w\rightarrow\infty$, but this
bound exists here by assumption.
\end{proof}
We note that the condition that (\ref{eq:separated potential}) imposes
on $h$ may be equivalently rewritten, in terms of $f$, as
\[
\limsup_{x\in\mathbb{T}^{d},m\rightarrow\infty}\frac{1}{mf_{m}}\left|m\frac{f_{mm}}{f_{m}}+1\right|^{2}<\infty.
\]
This condition, in particular, allows for $f$ to be combinations
of powers $m^{\alpha}$,$-m^{-\beta}$, exponentials $e^{m}$,$-e^{1/m}$,
and such typical examples, as long as one has the required blowup
near $m=0$ and as $m\rightarrow\infty$.

\section{Classical solutions\label{sec:Classical-solutions}}

To obtain classical solutions, it is necessary to have Hölder estimates
for the gradient of the solution in terms of the $C^{1}$ norm. The
following Lemma, which is merely a restatement of Theorems \ref{thm:Ladyzhenskaya interior}
and \ref{thm:Lieberman oblique boundary} in the context of the MFG
system, provides such an estimate.
\begin{lem}
\label{lem:a priori C1,gamma} Let $(u,m)\in C^{3}(\overline{Q_{T}})\times C^{2}(\overline{Q_{T}})$
be a solution to (\ref{eq:ellip}), and set $K=||u||_{C^{1}(\QT)}.$
Let $\mu_{K},\;\lambda_{K}>0$ be such that (\ref{eq:interior estimate, ellipticity and coeff bound})
holds in $\mathbb{T}_{K}^{d}$, and the conditions (\ref{eq:boundary muK})
and
\begin{equation}
\lambda_{K}\leq D_{q}B\cdot\nu\label{eq:oblique}
\end{equation}
hold in $\partial Q_{T,K}$. There exist constants $C>0,\;0<\gamma<1$,
with
\[
C=C\left(K,\frac{\mu_{K}}{\lambda_{K}}\right),\;\gamma=\gamma\left(K,\frac{\mu_{K}}{\lambda_{K}}\right),
\]
such that
\[
[Du]_{\gamma,\overline{Q_{T}}}\leq C.
\]
\end{lem}

\begin{proof}
The only thing to remark is that in order to apply Theorem \ref{thm:Lieberman oblique boundary},
it is necessary to verify that $\lambda_{K}$ can be chosen to satisfy
(\ref{eq:oblique}), or, in other words, that $N$ is indeed an oblique
boundary operator. This follows directly from (\ref{eq:boundary}),
since
\begin{align*}
D_{q}B(x,0,z,q)\cdot\nu(x,0)=-B_{s}(x,0,z,q)=1>0,\\
D_{q}B(x,T,z,q)\cdot\nu(x,T)=B_{s}(x,T,z,q)=g_{m}f_{w}^{-1}=\frac{g_{m}}{f_{m}}>0.
\end{align*}
Therefore, the result follows by applying Theorems \ref{thm:Ladyzhenskaya interior}
and \ref{thm:Lieberman oblique boundary} locally, and extracting
a finite subcover of $\QT$. The use of Theorem \ref{thm:Lieberman oblique boundary}
is particularly straighforward since the boundary of $Q_{T}$ is already
flat.
\end{proof}
The strategy to prove existence will be to use the nonlinear method
of continuity, by constructing an explicit homotopy $(Q^{\theta},N^{\theta})_{\theta\in[0,1]}$
between (\ref{eq:ellip}) and an elliptic problem that comes from
a much simpler MFG system, and trivially has a smooth solution. For
each $\theta\in[0,1]$ and each $(x,p,m)\in\mathbb{T}^{d}\times\mathbb{R}^{d}\times(0,\infty),$
define
\begin{align*}
H^{\theta}(x,p) & =\theta H(x,p)+(1-\theta)(\frac{1}{2}|p|^{2}+f(x,1)),\;g^{\theta}(x,m)=\theta g(x,m)+(1-\theta)m,\;m_{0}^{\theta}(x)=\theta m_{0}(x)+(1-\theta),
\end{align*}
and consider the family of MFG systems
\begin{equation}
\tag{\ensuremath{\text{MFG}_{\theta}}}\begin{cases}
-u_{t}+H^{\theta}(\cdot,D_{x}u)=f(\cdot,m),\;\;u(\cdot,T)=g^{\theta}(\cdot,m(\cdot,T))\\
m_{t}-\textrm{div}(mD_{p}H^{\theta}(\cdot,D_{x}u))=0,\;\;m(\cdot,0)=m_{0}^{\theta}(\cdot)
\end{cases}\label{eq:MFG theta}
\end{equation}
We observe that, when $\theta=0$, the unique solution is $(u,m)\equiv(1,1)$.
Let $(Q^{\theta}u$ $,N^{\theta}u)$ be the operators for the corresponding
elliptic problem associated to (\ref{eq:MFG theta}), and let $A^{\theta},$
$b^{\theta}$, and $B^{\theta}$ be their coefficients. The following
straightforward Lemma is a version of Theorem \ref{thm:Full C^1 a priori bound},
tailored to the family (\ref{eq:MFG theta}), that also includes the
Hölder estimates of Lemma \ref{lem:a priori C1,gamma}, and provides
a priori bounds that hold uniformly in $\theta$.
\begin{lem}
\label{lem:uniform in theta bounds}Assume that (\ref{eq:SE}) holds.
For each $\theta\in[0,1]$, let $(u^{\theta},m^{\theta})\in C^{3,\alpha}(\overline{Q_{T}})\times C^{2,\alpha}(\overline{Q_{T}})$
be a solution to (\ref{eq:MFG theta}). Then there exist constants
$C>0$ and $0<\gamma<1$, independent of $\theta$, such that
\[
||u^{\theta}||_{C^{1,\gamma}(\QT)}\leq C.
\]
\end{lem}

\begin{proof}
The strategy here is to apply Theorem \ref{thm:Full C^1 a priori bound}
and Lemma \ref{lem:a priori C1,gamma} to the corresponding MFG system
(\ref{eq:MFG theta}) that arises from the new data $H^{\theta},g^{\theta},m_{0}^{\theta}$,
to prove that those results lead to bounds that are uniform in $\theta$.
Let $\beta$ be defined by (\ref{eq:beta_K}), and, for each $\theta\in[0,1]$,
let $C_{0,\theta}$ and $0\leq\tau^{\theta}<1$ be any two constants
large enough that the inequalities (\ref{eq:strconv}), (\ref{qg1}),
(\ref{eq:qg2}), (\ref{eq:Hxxp, Hxpp bound}), (\ref{eq: fmx bound}),
(\ref{eq:fx, fxx bound}), (\ref{eq: H>=00003D|p|^2-C}), (\ref{eq: Hx, Hxx bound}),
(\ref{eq: Hxp bound}), and (\ref{eq: old C1}) all hold when $H$,
$g,$ $m_{0}$ are replaced by $H^{\theta}$, $g^{\theta},$ $m_{0}^{\theta}$.
Theorem \ref{thm:Full C^1 a priori bound} then yields constants $L_{\theta},L_{1,\theta},K_{\theta},K_{1,\theta}$,
with
\[
L_{\theta}=\left(L_{1,\theta},|g_{1}^{\theta}f_{0}^{-1}(L_{1,\theta})|,|g_{0}^{\theta}f_{1}^{-1}(-L_{1,\theta})|,(g_{0}^{\theta})^{-1}g_{1}^{\theta}f_{0}^{-1}(L_{1,\theta}),\frac{1}{(g_{1}^{\theta})^{-1}g_{0}^{\theta}f_{1}^{-1}(-L_{1,\theta})}\right),\;L_{1,\theta}=M(C_{0,\theta},T),
\]
\[
K_{\theta}=K_{\theta}(K_{1,\theta},\beta_{K_{1,\theta}}),\;K_{1,\theta}=K_{1,\theta}\left(L_{\theta},\frac{1}{T},\frac{1}{1-\tau^{\theta}},f_{0}\left(\frac{1}{L_{\theta}}\right)^{-}\right),
\]
such that
\[
||u^{\theta}||_{C^{0}(\overline{Q_{T}})}+||m^{\theta}(T)||_{C^{0}(\mathbb{T}^{d})}+\left\Vert \frac{1}{m^{\theta}(T)}\right\Vert _{C^{0}(\mathbb{T}^{d})}\leq L_{\theta},\;||Du^{\theta}||_{C^{0}(\overline{Q_{T}})}\leq K_{\theta}.
\]
The goal is now to show that $L_{\theta}$, $K_{\theta}$ may be chosen
independently of $\theta$. First we prove that this is true for $C_{0,\theta}$
and $\tau^{\theta}.$ Conditions (\ref{eq: fmx bound}) and (\ref{eq:fx, fxx bound})
trivially hold for the same $C_{0}$ and the new $H^{\theta}$, $g^{\theta},$
$m_{0}^{\theta}$, because the functions $H$, $g,$ $m_{0}$ do not
appear in those inequalities. Since the map $H^{0}(p,x)=\frac{1}{2}|p|^{2}+f(x,1)$
satisfies $D_{p}H^{0}\equiv p$, it also satisfies (\ref{eq:strconv}),
(\ref{eq:qg2}), (\ref{eq:Hxxp, Hxpp bound}), and (\ref{eq: Hxp bound}),
with $C_{0}$ being replaced by a universal constant. Thus, since
$H^{\theta}$ is a convex combination of $H^{0}$ and $H$, these
inequalities still hold for $H^{\theta},$ when $C_{0}$ is replaced
by a convex combination of $C_{0}$ and a universal constant. By the
same reasoning, conditions (\ref{qg1}), (\ref{eq: H>=00003D|p|^2-C}),
and (\ref{eq: Hx, Hxx bound}) hold for $H^{\theta}$ after replacing
$C_{0}$ with a convex combination of $C_{0}$ and a constant depending
only on $C_{0}$ and $||f(\cdot,1)||_{C^{2}(\mathbb{T}^{d})}\leq C_{0}$.
Only condition (\ref{eq: old C1}) is left to consider, namely
\begin{align}
||\chi(\cdot,0)||_{C^{0}(\mathbb{T}^{d})}+||m_{0}^{\theta}||_{C^{1}(\mathbb{T}^{d})}+||f||_{C^{2}(\mathbb{T}^{d}\times[\min m_{0}^{\theta},\max m_{0}^{\theta}])}\leq C_{0,\theta}.\label{eq:C1kappa def}
\end{align}
The first term is already independent of $\theta,$ whereas, noticing
that $\min m_{0}\leq1\leq\max m_{0}$ and $|D_{x}m_{0}^{\theta}|=\theta|D_{x}m_{0}|$,
\begin{align*}
||m_{0}^{\theta}||_{C^{1}(\mathbb{T}^{d})}+||f||_{C^{2}(\mathbb{T}^{d}\times[\min m_{0}^{\theta},\max m_{0}^{\theta}])} & \leq||m_{0}||_{C^{1}(\mathbb{T}^{d})}+||f||_{C^{2}(\mathbb{T}^{d}\times[\min m_{0},\max m_{0}])}\leq C_{0}.
\end{align*}
Thus, one may select
\begin{equation}
C_{0,\theta}=C_{0,\theta}(C_{0}),\;\tau^{\theta}=\tau,\label{eq: C0kappa ind}
\end{equation}
and consequently
\[
L_{1,\theta}=L_{1,\theta}(C_{0,\theta},T)=L_{1,\theta}(C_{0},T):=L_{1}.
\]
Now, by definition,
\begin{gather}
g_{0}^{\theta}(m)=\theta g_{0}+(1-\theta)m,\;\;\;g_{1}^{\theta}(m)=\theta g_{1}+(1-\theta)m.\label{eq:Psi_k, psi_k}
\end{gather}
Therefore,
\begin{align}
|g_{0}^{\theta}f_{1}^{-1}(-L_{1})| & \leq\max(|g_{0}f_{1}^{-1}(-L_{1})|,f_{1}^{-1}(-L_{1}))\leq\max(|g_{0}f_{1}^{-1}(-L_{1})|,f_{0}^{-1}(L_{1})),\label{eq:kappa ind 1}
\end{align}
and similarly,
\begin{align}
|g_{1}^{\theta}f_{0}^{-1}(L_{1}) & |=|\theta g_{1}f_{0}^{-1}(L_{1})+(1-\theta)f_{0}^{-1}(L_{1})|\leq\max(|g_{1}f_{0}^{-1}(L_{1})|,f_{0}^{-1}(L_{1})).\label{eq: kappa ind 2}
\end{align}
On the other hand, the following inequalities hold:

\begin{equation}
(g_{0}^{\theta})^{-1}g_{1}^{\theta}\leq g_{0}^{-1}g_{1},\;\;\;g_{1}^{-1}g_{0}\leq(g_{1}^{\theta})^{-1}g_{0}^{\theta}.\label{eq:psiPsi}
\end{equation}
Indeed, by (\ref{eq:Psi_k, psi_k}), 
\begin{align*}
g_{0}^{\theta}g_{0}^{-1}g_{1} & =\theta g_{0}g_{0}^{-1}g_{1}+(1-\theta)g_{0}^{-1}g_{1}\geq\theta g_{1}+(1-\theta)g_{0}^{-1}g_{0}=g_{1}^{\theta},
\end{align*}
which shows the first inequality in (\ref{eq:psiPsi}), with the second
one following in the same fashion. Now, (\ref{eq:psiPsi}) yields
\begin{align}
(g_{0}^{\theta})^{-1}g_{1}^{\theta}f_{0}^{-1}(L_{1}) & \leq g_{0}^{-1}g_{1}f_{0}^{-1}(L_{1}),\;\;\;\frac{1}{(g_{1}^{\theta})^{-1}g_{0}^{\theta}f_{1}^{-1}(-L_{1})}\leq\frac{1}{g_{1}^{-1}g_{0}f_{1}^{-1}(-L_{1})}.\label{eq: kappa ind 3}
\end{align}
Thus, (\ref{eq:kappa ind 1}), (\ref{eq: kappa ind 2}), and (\ref{eq: kappa ind 3})
yield
\begin{align*}
L_{\theta} & =L_{\theta}\left(M,|g_{1}f_{0}^{-1}(L_{1})|,|g_{0}f_{1}^{-1}(L_{1})|,g_{0}^{-1}g_{1}f_{0}^{-1}(L_{1}),\frac{1}{g_{1}^{-1}g_{0}f_{1}^{-1}(-L_{1})},f_{0}^{-1}(L_{1})\right):=L,
\end{align*}
and
\begin{align*}
K_{1,\theta} & =K_{1,\theta}\left(L,\frac{1}{T},\frac{1}{1-\tau},f_{0}\left(\frac{1}{L}\right)^{-}\right):=K_{1},\;K_{\theta}=K_{\theta}(K_{1},\beta_{K_{1}}):=K.
\end{align*}
Next, we obtain the gradient Hölder estimate with the help of Lemma
\ref{lem:a priori C1,gamma}. We remark that the operator $(Q^{\theta},N^{\theta})$
is clearly elliptic and oblique, because it comes from (\ref{eq:MFG theta}).
Moreover, since $A^{\theta},b^{\theta},$ and $B^{\theta}$ and their
derivatives are, respectively, continuous functions of $(x,t,z,p,s,\theta)$
on the compact sets $\overline{Q}_{T,K}\times[0,1]$ and $\partial Q_{T,K}\times[0,1]$,
it follows that there exist constants $\mu_{L+K}>0,\;\lambda_{L+K}>0$,
independent of $\theta$, satisfying (\ref{eq:interior estimate, ellipticity and coeff bound})
in $(\mathbb{T}^{d})_{L+K}$, and (\ref{eq:boundary muK}), (\ref{eq:oblique})
in $\partial Q_{T,L+K}$, when the operators $(Q,N)$ are replaced
by $(Q^{\theta},N^{\theta})$. Lemma \ref{lem:a priori C1,gamma}
then yields constants $C>0$, $0<\gamma<1$, independent of $\theta$,
such that
\[
[Du^{\theta}]_{\gamma,\QT}\leq C.
\]
\end{proof}
With the help of this uniform estimate, the main theorem for the strictly
elliptic problem may now be proved.
\begin{proof}[Proof of Theorem \ref{thm:smoothsols}]
The uniqueness part of the statement is an immediate consequence
of the standard Lasry-Lions monotonicity method, and will be omitted.
We define the Banach spaces
\[
E=C^{3,\alpha}(\overline{Q_{T}}),\;F=C^{1,\alpha}(\overline{Q_{T}})\times C^{2,\alpha}(\partial Q_{T}),
\]
and the continuously differentiable operator $S:E\times[0,1]\rightarrow F$
by
\[
S(u,\theta)=(Q^{\theta}u,N^{\theta}u),\;(u,\theta)\in E\times[0,1].
\]
The partial Fréchet derivative of $S$ with respect to the variable
$u$ at the point $(u,\theta$) is the corresponding linearization,
for fixed $\theta$, of the differential operator $(Q^{\theta},N^{\theta})$,
namely $(L_{(u,\theta)}^{1},L_{(u,\theta)}^{2})$, where
\begin{align*}
L_{(u,\theta)}^{1}(w)= & -\tr(A^{\theta}(x,Du)D^{2}w)-D_{q}\tr(A^{\theta}(x,Du)D^{2}u)\cdot Dw+D_{q}b^{\theta}(x,Du)\cdot Dw,\\
L_{(u,\theta)}^{2}(w)= & \begin{cases}
-w_{t}+D_{p}H^{\theta}(x,D_{x}u)\cdot D_{x}w & \text{if }t=0,\\
\frac{g_{m}^{\theta}}{f_{m}}(w_{t}-D_{p}H^{\theta}\cdot D_{x}w)+w & \text{if }t=T.
\end{cases}
\end{align*}
For fixed $(u,\theta)\in E\times[0,1]$, the linear operator $L_{(u,\theta)}^{1}$
is uniformly elliptic and the linear boundary operator $L_{(u,\theta)}^{2}$
is oblique. Moreover, the homogeneous problem $(L_{(u,\theta)}^{1}w,L_{(u,\theta)}^{2}w)=(0,0)$
has the form
\[
-\tr(\tilde{A}(x,t)D^{2}w)+\tilde{b}(x,t)\cdot Dw=0\text{ in }Q_{T},\;\;\;\tilde{B}(x,t)\cdot Dw+\tilde{c}(x,t)w=0\text{ on }\partial Q_{T},
\]
where $\tilde{B}\cdot\nu>0$, $\tilde{c}\geq0$ and $\tilde{c}\not\equiv0$,
which implies that it has only the trivial solution in $C^{3,\alpha}(\overline{Q_{T}})$.
Hence, by the standard Fredholm alternative for linear oblique problems
(see \cite{GilbargTrudinger}), the operator $(L_{(u,\theta)}^{1},L_{(u,\theta)}^{2})$
is invertible in $C^{3,\alpha}(\overline{Q_{T}})$. The infinite-dimensional
implicit function theorem then implies that the set
\[
D=\{\theta\in[0,1]:\text{the equation \ensuremath{S(u,\theta)=(0,0)} has a unique solution \ensuremath{u\in C^{3,\alpha}(\overline{Q_{T}})}}\}
\]
is open in $[0,1]$. 

The next step is to show that $D$ is also closed. Let $\{\theta_{n}\}\subset D$
be a sequence such that $\theta_{n}\rightarrow\theta\in[0,1]$, and
let $\{u_{n}\}\subset E$ be the corresponding sequence of solutions
to $S(u_{n},\theta_{n})=(0,0)$. By Lemma \ref{lem:uniform in theta bounds},
there exist numbers $C>0,\;0<\gamma<1$, independent of $n$, such
that 
\[
||u_{n}||_{C^{1,\gamma}(\overline{Q_{T}})}\leq C.
\]
The Arzelà--Ascoli Theorem implies that, up to a subsequence, there
exists $u\in C^{1,\gamma}(\overline{Q_{T}})$ such that $u_{n}\rightarrow u$
in $C^{1}(\overline{Q_{T}})$. By Theorem \ref{thm:Fiorenza convergence},
it follows that $u\in C^{2,\alpha}(\overline{Q_{T}})$, $u_{n}\rightarrow u$
in $C^{2,\alpha}(\overline{Q_{T}})$, and $S(u,\theta)=0$. In particular,
the $u_{n}$ are uniformly bounded in $C^{2,\alpha}(\overline{Q_{T}})$.
Now, given $i\in\{1,\ldots,d\}$, differentiating the equation $(Q^{\theta_{n}}(u_{n}),N^{\theta_{n}}(u_{n}))=(0,0)$
yields, for $w=D_{x_{i}}u_{n}$,
\begin{align*}
L_{(u_{n},\theta_{n})}^{1}w= & \tr(A_{x_{i}}^{\theta_{n}}(x,Du_{n})D^{2}u_{n})-b_{x_{i}}^{\theta_{n}}(x,Du_{n}),\\
L_{(u_{n},\theta_{n})}^{2}w= & \begin{cases}
D_{x_{i}}(f(x,m_{0}^{\theta_{n}}(x)))-H_{x_{i}}^{\theta_{n}} & \text{if }t=0,\\
\frac{g_{m}^{\theta_{n}}}{f_{m}}(H_{x_{i}}^{\theta_{n}}-f_{x_{i}})+g_{x_{i}}^{\theta_{n}} & \text{if }t=T.
\end{cases}
\end{align*}
Therefore, by Theorem \ref{thm:oblique Schauder}, there exists $C>0$,
independent of $n$, such that
\[
||w||_{C^{2,\alpha}(\overline{Q_{T}})}\leq C,
\]
implying that $D_{x}u_{n}$ is bounded in $C^{2,\alpha}(\overline{Q_{T}}).$
In particular, $u_{n}|_{\partial Q_{T}}$ is bounded in $C^{3,\alpha}(\partial Q_{T})$,
and by the standard Schauder theory for the Dirichlet problem, $u_{n}$
is therefore bounded in $C^{3,\alpha}(\overline{Q_{T}})$. Consequently,
$u\in C^{3,\alpha}(\overline{Q_{T}})$ and $\theta\in D,$ proving
that $D$ is closed. Since $0\in D$, it follows that $D=[0,1]$,
which completes the proof.
\end{proof}
The next theorem is the corresponding variant of Theorem \ref{thm:weaksols}
for the case of a fast-growing $f$, which follows from the estimates
in Theorem \ref{thm:V(x) a priori C^1}.
\begin{thm}
\label{thm:V(x) smoothsols} If condition (\ref{growth chi}) is replaced
by (\ref{eq:separated potential}), the conclusion of Theorem \ref{thm:smoothsols}
holds.
\end{thm}

\begin{proof}
All of the results in this section follow in this case by simply replacing
the use of Theorem \ref{thm:Full C^1 a priori bound} by Theorem \ref{thm:V(x) a priori C^1}.
\end{proof}

\section{Weak solutions\label{sec:Weak-solutions}}

In this section we develop the theory of weak solutions, for the case
where the strict ellipticity condition (\ref{eq:SE}) fails to hold.
We begin by stating the definition of weak solution that will be used,
which is in direct analogy with the one used in \cite{Cardalaguiet,CardaliaguetGraber,CardaliaguetGraberPorrettaTonon}
to study the degenerate case in which $g_{m}\equiv0$.
\begin{defn}[Definition of weak solution]
\label{def:weaksol def}A pair $(u,m)\in\text{{BV}}(Q_{T})\times L_{+}^{\infty}(Q_{T})$
is called a weak solution to (\ref{eq:mfg}) if the following conditions
hold:
\end{defn}

\begin{enumerate}
\item[(i)] $D_{x}u\in L^{2}(Q_{T}),u\in L^{\infty}(Q_{T}),$ $m\in C^{0}([0,T];H^{-1}(\mathbb{T}^{d}))$,
$m(\cdot,T)\in L^{\infty}(\mathbb{T}^{d})$.
\item[(ii)] $u$ satisfies the HJ inequality
\[
-u_{t}+H(\cdot,D_{x}u)\leq f(\cdot,m)\;\;\text{ in }Q_{T},\;\;\;u(\cdot,T)=g(\cdot,m(\cdot,T))\;\;\text{ in }\mathbb{T}^{d},
\]
in the distributional sense, with $u(\cdot,T)=g(\cdot,m(\cdot,T))$
in the sense of traces.
\item[(iii)] $m$ satisfies the continuity equation
\begin{equation}
m_{t}-\text{div}(mD_{p}H(\cdot,D_{x}u))=0\text{ in }Q_{T},\;\;\;m(\cdot,0)=m_{0}\text{ in }\mathbb{T}^{d}\label{eq: FP weak}
\end{equation}
in the distributional sense, with $m(\cdot,0)=m_{0}$ in $H^{-1}(\mathbb{T}^{d})$.
\item[(iv)] The following identity holds:
\begin{multline}
\intQ m(x,t)(H(x,D_{x}u)-D_{p}H(x,D_{x}u)\cdot D_{x}u-f(x,m))dxdt\\
=\int_{\mathbb{T}^{d}}(m(x,T)g(x,m(x,T))-m_{0}(x)u(x,0))dx.\label{eq:identity weak}
\end{multline}
\end{enumerate}
The solutions to the degenerate elliptic problem will be obtained
as a ``vanishing viscosity'' limit of MFG systems satisfying (\ref{eq:SE}),
in the following sense. Assuming that (\ref{eq:DE}) holds, we consider,
for $\epsilon>0$, the system

\begin{equation}
\tag{\ensuremath{\text{MFG}_{\epsilon}}}\begin{cases}
-u_{t}^{\epsilon}+H(\cdot,D_{x}u^{\epsilon})=f(\cdot,m^{\epsilon})+\epsilon\log(m^{\epsilon}),\;\;\;u(\cdot,T)=g(\cdot,m^{\epsilon}(\cdot,T))\\
m_{t}^{\epsilon}-\text{\emph{\emph{div}}}(m^{\epsilon}D_{p}H(\cdot,D_{x}u^{\epsilon}))=0,\;\;\;m^{\epsilon}(0)=m_{0},
\end{cases}\label{eq:MFG epsilon}
\end{equation}
Since (\ref{eq:MFG epsilon}) is strictly elliptic, by Theorem \ref{thm:smoothsols}
it has a unique solution $(u^{\epsilon},m^{\epsilon})\in C^{3}(\overline{Q_{T}})\times C^{2}(\overline{Q_{T}})$.
The only missing ingredient necessary to obtain a solution as the
limit when $\epsilon\rightarrow0$ is the following minor modification
of Lemma \ref{lem: gradient a priori bound}, which provides a global,
a priori upper bound for the density that is independent of the size
of $\frac{1}{\min(m(T))}$.
\begin{lem}
\renewcommand*{\theHequation}{notag13.\theequation}

\label{lem: f a priori bound weak} Assume that (\ref{eq:SE}) holds.
If $(u,m)\in C^{3}(\overline{Q_{T}})\times C^{2}(\overline{Q_{T}})$
is a solution to (\ref{eq:mfg}), then
\[
\max_{\overline{Q_{T}}}f(\cdot,m(\cdot,\cdot))\leq C,
\]
where 
\[
C=C\left(C_{0},T,\frac{1}{T},\frac{1}{1-\tau},||u||_{C^{0}(\QT)},\max_{\mathbb{T}^{d}}m(T),f_{1}(\max_{\mathbb{T}^{d}}m(T))^{+},\left\Vert \frac{1}{\chi}\right\Vert _{C^{0}(\mathbb{T}^{d}\times[1,\infty))},||h_{w}||_{C^{0}(\mathbb{T}^{d}\times[1,\infty))}\right).
\]
\end{lem}

\begin{proof}
The argument is a simple variant of the proof of Lemma \ref{lem: gradient a priori bound}.
Let $v$ and $\utilde$ have the same meaning as in said proof, with
$0<c_{1}<1$ once more being a free parameter, set $\tilde{v}=-u_{t}+v=f+\frac{c_{1}}{2}\utilde^{2}$,
and let $(x_{0},t_{0})$ be a point where $\tilde{v}$ achieves its
maximum value. 

\textbf{Case 1. }If $t_{0}\in\{0,T\}$, then
\begin{align*}
\tilde{v}(x_{0},t_{0}) & =f(x_{0},m(x_{0},t_{0}))+\frac{c_{1}}{2}\utilde^{2}\leq\max(f_{1}(\max m(T)),f_{1}(\max m_{0}))+C\leq C.
\end{align*}

\textbf{Case 2. }Assume next that $0<t_{0}<T$. Without loss of generality,
it may be assumed that $f=f(x_{0},m(x_{0},t_{0}))\geq1$, because
otherwise there would be nothing to prove. Therefore, using (\ref{eq:qg2}),
since
\begin{align*}
|-u_{t}+D_{p}H\cdot D_{x}u| & \geq-u_{t}+D_{p}H\cdot D_{x}u\geq f+H,
\end{align*}
it follows, by (\ref{eq: H>=00003D|p|^2-C}), that
\begin{equation}
|-u_{t}+D_{p}H\cdot D_{x}u|^{2}\geq\frac{1}{2}|f|^{2}+\frac{1}{2C_{0}^{2}}|D_{x}u|^{4}-C.\label{eq:dominant term weak}
\end{equation}
Now, by Corollary \ref{cor:utbound}, $L_{u}(\tilde{v})=L_{u}(v)$,
so, as in (\ref{eq:L_u(v) quote weak}), Lemma \ref{lem: L(v1) and L(v2)}
yields
\begin{multline*}
0\leq L_{u}(\tilde{v})\leq\frac{-1}{2C_{0}}|-D_{x}u_{t}+D_{p}HD_{xx}^{2}u|^{2}-\frac{\chi}{2C_{0}^{2}}|D_{xx}^{2}u|^{2}-\frac{c_{1}}{2}|-u_{t}+D_{p}H\cdot D_{x}u|^{2}-\frac{c_{1}}{C_{0}}\chi|D_{x}u|^{2}\\
+C(1+|D_{x}u|^{3+\tau}+\chi(1+|D_{x}u|^{1+\tau})+\chi^{(1+\tau)/2}|D_{x}u|^{2}+|f|(1+|D_{x}u|^{1+\tau})+|D_{x}u|^{2}|f|^{(1+\tau)/2})\\
+Cc_{1}(\chi|D_{xx}^{2}u|^{2}+|-D_{x}u_{t}+D_{xx}^{2}uD_{p}H|^{2}).
\end{multline*}
Thus, by (\ref{growth chi}), (\ref{eq:dominant term weak}), and
the fact that $f\geq1$,
\begin{multline*}
0\leq\frac{-1}{2C_{0}}|-D_{x}u_{t}+D_{p}HD_{xx}^{2}u|^{2}-\frac{\chi}{2C_{0}^{2}}|D_{xx}^{2}u|^{2}-\frac{c_{1}}{4}\left(\frac{1}{C_{0}^{2}}|D_{x}u|^{4}+f{}^{2}\right)-\frac{c_{1}}{C_{0}}\chi|D_{x}u|^{2}\\
+C(1+|D_{x}u|^{3+\tau}+(1+f)(1+|D_{x}u|^{1+\tau})+(1+f)^{(1+\tau)/2}|D_{x}u|^{2}+f(1+|D_{x}u|^{1+\tau})\\
+|D_{x}u|^{2}f{}^{(1+\tau)/2})+Cc_{1}(\chi|D_{xx}^{2}u|^{2}+|-D_{x}u_{t}+D_{xx}^{2}uD_{p}H|^{2}).
\end{multline*}
Once more, as in Lemma \ref{lem: gradient a priori bound}, fix $c_{1}$
such that $c_{1}<\frac{1}{4C_{0}C(1+C_{0})}$, where the constant
$C$ is as in the previous line, obtaining
\begin{align*}
\frac{c_{1}}{4}\left(\frac{1}{C_{0}}|D_{x}u|^{4}+f{}^{2}\right)\leq & C(1+|D_{x}u|^{3+\tau}+(1+f)(1+|D_{x}u|^{1+\tau})+(1+f)^{(1+\tau)/2}|D_{x}u|^{2}\\
 & +f(1+|D_{x}u|^{1+\tau})+|D_{x}u|^{2}f{}^{(1+\tau)/2}).
\end{align*}
The left- and right-hand sides have, respectively, degree $4$ and
degree $3+\tau<4$ in the non-negative variables $(|D_{x}u|,\sqrt{f})$,
thus $|D_{x}u|^{2}+f\leq C,$ and, in particular, it follows that
\[
\tilde{v}(x_{0},t_{0})\leq C.
\]
\qedhere
\end{proof}
We now obtain several a priori bounds for $(u^{\epsilon},m^{\epsilon})$
that are independent of $\epsilon$.
\begin{lem}
\label{lem: weak Linfty estimates}Assume $(\ref{eq:DE})$, and let
$(u^{\epsilon},m^{\epsilon})\in C^{3,\alpha}(\overline{Q_{T}})\times C^{2,\alpha}(\overline{Q_{T}})$
be the solution to (\ref{eq:MFG epsilon}). Then there exist constants
$L,L_{1},C,$ and $C_{1}$, with

\[
L=L(L_{1},|g_{1}f_{0}^{-1}(L_{1})|,g_{0}^{-1}g_{1}f_{0}^{-1}(L_{1})),\;L_{1}=(C_{0},T,\text{\ensuremath{|\min_{\mathbb{T}^{d}}f(\cdot,0)|}},|\min_{\mathbb{T}^{d}}g(\cdot,0)|),
\]
\[
C=C(C_{1},f_{1}(C_{1})^{+}),\;C_{1}=C_{1}\left(L,\frac{1}{T},\frac{1}{1-\tau},\left\Vert \frac{1}{\chi}\right\Vert _{C^{0}(\mathbb{T}^{d}\times[1+\max_{\mathbb{T}^{d}}f^{+}(\cdot,0),\infty))},||h_{w}||_{C^{0}(\mathbb{T}^{d}\times[1+\max_{\mathbb{T}^{d}}f^{+}(\cdot,0),\infty))}\right),
\]
such that, for every $\epsilon\leq\frac{1}{C}$,
\begin{gather}
||u^{\epsilon}||_{C^{0}(\QT)}\leq L,\;\;\;||m^{\epsilon}||_{C^{0}(\QT)}+||u_{t}^{-,\epsilon}||_{C^{0}(\QT)}\leq C,\label{eq:u epsilon bdd}\\
||u_{t}^{\epsilon}||_{L^{1}(Q_{T})}+||D_{x}u^{\epsilon}||_{L^{2}(Q_{T})}+\epsilon||\log m^{\epsilon}||_{L^{1}(Q_{T})}\leq C.\label{eq:energy bounds}
\end{gather}
\end{lem}

\begin{proof}
By replacing $f,\;H$ with $f-C,\;H-C$, for $C$ depending only on
$C_{0}$, there is no loss of generality in assuming $f(\cdot,0)\leq-1$.
It is readily seen that conditions (\ref{growth chi}), (\ref{eq: fmx bound}),
(\ref{eq:fx, fxx bound}), and (\ref{eq: old C1}) hold for $f^{\epsilon}=f+\epsilon\log(\cdot)$,
uniformly in $\epsilon,$ up to increasing $C_{0}$ by a finite value.
By Lemma \ref{lem:aprioriu}, there exists $L_{1}$ such that
\[
g_{0}(f_{1}^{\epsilon})^{-1}(-L_{1})-L_{1}(e^{L_{1}T}-e^{L_{1}t})\leq u^{\epsilon}\leq g_{1}(f_{0}^{\epsilon})^{-1}(L_{1})+L_{1}(e^{L_{1}T}-e^{L_{1}t}).
\]
Now, (\ref{eq:DE}) implies that 
\[
-L-L_{1}(e^{L_{1}T}-e^{L_{1}t})\leq u^{\epsilon}.
\]
On the other hand, if $L_{1}>f_{0}(1)$, then $f_{0}{}^{-1}(L_{1})>1$,
hence
\[
f_{0}^{\epsilon}(f_{0}{}^{-1}(L_{1}))=L_{1}+\epsilon\log(f_{0}^{-1}(L_{1}))\geq L_{1},
\]
which implies
\begin{equation}
g_{1}f_{0}^{-1}(L_{1})\geq g_{1}(f_{0}^{\epsilon})^{-1}(L_{1}).\label{eq:g1 fep_0 ^-1  weak}
\end{equation}
Consequently,
\[
u^{\epsilon}\leq g_{1}f_{0}{}^{-1}(L_{1})+L_{1}(e^{L_{1}T}-e^{L_{1}t}),
\]
which proves the first inequality in (\ref{eq:u epsilon bdd}). Now,
by Corollary \ref{cor:apriorim} and (\ref{eq:g1 fep_0 ^-1  weak}),
\[
m^{\epsilon}(\cdot,T)\leq L.
\]
Thus, Lemma \ref{lem: f a priori bound weak} implies that
\begin{equation}
f(\cdot,m^{\epsilon})+\epsilon\log m^{\epsilon}(\cdot,\cdot)=f^{\epsilon}(\cdot,m^{\epsilon})\leq C,\label{eq:-1}
\end{equation}
and (\ref{eq:f polynomial growth}) yields, for $C$ depending only
on $||\frac{1}{\chi}||_{C^{0}(\mathbb{T}^{d}\times[1,\infty))}$ and
$|\min f(\cdot,0)|$,
\[
\frac{1}{C}\log(m^{\epsilon}(\cdot,\cdot))-C\leq f(\cdot,m^{\epsilon}(\cdot,\cdot)).
\]
Therefore, by (\ref{eq:-1}) we conclude that $||m^{\epsilon}||_{C^{0}(\QT)}\leq C$
for large enough $C$ and small enough $\epsilon$. The lower bound
for $u_{t}^{\epsilon}$ is simply a consequence of (\ref{eq:-1}),
the relation $-u_{t}^{\epsilon}+H=f^{\epsilon},$ and the fact that
$H$ is bounded below. This completes the proof of (\ref{eq:u epsilon bdd}). 

Now, integrating the HJ equation for $u^{\epsilon}$ yields
\begin{equation}
\intQ H(\cdot,D_{x}u^{\epsilon})+\intQ\epsilon(\log m^{\epsilon})^{-}=\intQ(f(\cdot,m^{\epsilon})+\epsilon(\log m^{\epsilon})^{+})+\intO(u^{\epsilon}(T)-u^{\epsilon}(0)).\label{eq:weak inter}
\end{equation}
It then follows from (\ref{eq:u epsilon bdd}) and (\ref{eq: H>=00003D|p|^2-C})
that $||D_{x}u^{\epsilon}||_{L^{2}(Q_{T})}$ and $\epsilon||\log m^{\epsilon}||_{L^{1}(Q_{T})}$
are bounded. Finally, \[\left|\intQ u_{t}^{\epsilon}\right|=\left|\int_{\mathbb{T}^d}u^{\epsilon}(T)-u^{\epsilon}(0)\right|\leq C,\]
and since $u_{t}^{\epsilon}$ is bounded below, this proves (\ref{eq:energy bounds}).
\end{proof}
After extracting a subsequence, Lemma \ref{lem: weak Linfty estimates}
implies the existence of $(u,m)\in\text{{BV}}(Q_{T})\times L_{+}^{\infty}(Q_{T})$
such that, as $\epsilon\rightarrow0$,
\[
u^{\epsilon}\rightarrow u\text{ in }L^{1}(Q_{T})\text{ and pointwise a.e.,\text{ \;\;}\ensuremath{m^{\epsilon}\weak m}\text{ in }\ensuremath{L^{\infty}}(\ensuremath{Q_{T}}), }
\]
\vspace*{-0.75cm}

\begin{equation}
u_{t}^{\epsilon}\weak u_{t}\text{ in }C^{0}(\overline{Q_{T}})^{*},\;\;\;D_{x}u^{\epsilon}\rightharpoonup D_{x}u\text{ in }L^{2}(Q_{T}).\label{eq:u conv in BV}
\end{equation}
We now show that, up to a further subsequence, this convergence can
be strengthened.
\begin{lem}
\label{lem:convergence lemma weak}Assume that (\ref{eq:DE}) holds,
let $(u^{\epsilon},m^{\epsilon})\in C^{3}(\overline{Q_{T}})\times C^{2}(\overline{Q_{T}})$
be the solution to (\ref{eq:MFG epsilon}), and let $(u,m)\in BV(Q_{T})\times L_{+}^{\infty}(Q_{T})$
be a subsequential limit. Then, up to extracting a subsequence, 
\begin{equation}
m^{\epsilon}\rightarrow m\text{ in }C^{0}([0,T],H^{-1}(\mathbb{T}^{d}))\text{ and a.e. in }Q_{T},\label{eq:m conv}
\end{equation}
\begin{equation}
m^{\epsilon}(\cdot,T)\rightarrow m(\cdot,T),\;\;\;u^{\epsilon}(\cdot,T)\rightarrow u(\cdot,T)\text{ a.e. in \ensuremath{\mathbb{T}^{d}}, }\label{eq:term cond conv}
\end{equation}
\begin{align}
D_{x}u^{\epsilon} & \rightarrow D_{x}u\text{ in }L_{m}^{2}(Q_{T})\text{ and a.e. in }\{m>0\},\;\;m^{\epsilon}|D_{x}u^{\epsilon}|^{2}\rightarrow m|D_{x}u|^{2}\text{ in }L^{1}(Q_{T})\text{ and a.e. in }Q_{T},\label{eq: Du conv  L^2m}
\end{align}
\begin{equation}
\epsilon\log m^{\epsilon}\rightarrow0\text{ in }L_{m}^{1}(Q_{T}).\label{eq:epsilon log m conv}
\end{equation}
Moreover, $u(\cdot,T)=g(\cdot,m(\cdot,T))$ and $m\in C^{0,\frac{1}{2}}([0,T],H^{-1}(\mathbb{T}^{d}))$,
with

\[
[m]_{\frac{1}{2},[0,T],H^{-1}}\leq||mD_{p}H(\cdot,D_{x}u)||_{L^{2}(Q_{T})},
\]
and, for almost every $s\in[0,T]$, including $s=0$, 
\begin{align}
\int_{s}^{T}\intO m(x,t)(H(x,D_{x}u)-D_{p}H(x,D_{x}u)\cdot D_{x}u-f(x,m))dxdt\nonumber \\
=\int_{\mathbb{T}^{d}}(m(x,T)g(x,m(x,T))-m(x,s)u(x,s))dx.\label{eq:weak identity for all s}
\end{align}
\end{lem}

\begin{proof}
Let $\epsilon,\epsilon'>0$. We employ the standard Lasry-Lions method
with the pairs $(u^{\epsilon},m^{\epsilon})$ and $(u^{\epsilon'},m^{\epsilon'})$.
Namely, subtracting the HJ and continuity equations from each other,
respectively, yields a new system for $u^{\epsilon}-u^{\epsilon'}$
and $m^{\epsilon}-m^{\epsilon'}$. Multiplying the equations satisfied
by $u^{\epsilon}-u^{\epsilon'}$ and $m^{\epsilon}-m^{\epsilon'}$,
respectively, by $m^{\epsilon}-m^{\epsilon'}$ and $u^{\epsilon}-u^{\epsilon'}$,
and then integrating over $Q_{T}$, leads to the identity
\[
M_{\epsilon,\epsilon'}+M_{\epsilon',\epsilon}+M_{g}+M_{f}+K_{\epsilon,\epsilon'}^{-}+K_{\epsilon',\epsilon}^{-}=K+K_{\epsilon,\epsilon'}^{+}+K_{\epsilon',\epsilon}^{+},
\]
where
\[
M_{\epsilon,\epsilon'}=\intQ m^{\epsilon}(H^{\epsilon'}-H^{\epsilon}-D_{p}H^{\epsilon}\cdot(D_{x}u^{\epsilon'}-D_{x}u^{\epsilon})),
\]
\[
M_{g}=\intO(g(\cdot,m^{\epsilon}(\cdot,T))-g(\cdot,m^{\epsilon'}(\cdot,T))(m^{\epsilon}(\cdot,T)-m^{\epsilon'}(\cdot,T)),\;M_{f}=\intQ(f(\cdot,m^{\epsilon})-f(\cdot,m^{\epsilon'}))(m^{\epsilon}-m^{\epsilon'}),
\]
\[
K_{\epsilon,\epsilon'}^{\pm}=\intQ\epsilon'(\log m^{\epsilon'})^{\pm}m^{\epsilon},\;K=\intQ-(\epsilon'\log m^{\epsilon'}m^{\epsilon'}+\epsilon\log m^{\epsilon}m^{\epsilon}),
\]
and $H^{\epsilon}=H(\cdot,D_{x}u^{\epsilon})$, $D_{p}H^{\epsilon}=D_{p}H(\cdot,D_{x}u^{\epsilon}).$
By (\ref{eq:u epsilon bdd}), and the fact that the map $m\mapsto m\log m$
is bounded below, $K\rightarrow0$ as $\epsilon,\epsilon'\rightarrow0$.
Similarly, (\ref{eq:u epsilon bdd}) implies that the quantity $\log m^{\epsilon'}m^{\epsilon}$
is bounded above, and thus $K_{\epsilon,\epsilon'}^{+}\rightarrow0$
as $\epsilon,\epsilon'\rightarrow0.$ Since, by monotonicity and convexity,
each term on the left-hand side of the equation is non-negative, all
of them converge to zero. Every claim of a.e. convergence in what
follows is tacitly meant to hold after extracting a subsequence. Since
$M_{f}\rightarrow0$, we have $m^{\epsilon}\rightarrow m$ a.e. in
$Q_{T}$. Similarly, since $M_{g}\rightarrow0$, there exists a function
$m_{T}\in L^{\infty}(\mathbb{T}^{d})$ such that $m^{\epsilon}(T)\rightarrow m_{T}$
a.e. in $\mathbb{T}^{d}$. Now, using (\ref{eq:strconv}), one obtains
\[
\frac{1}{2C_{0}}\intQ(m^{\epsilon}+m^{\epsilon'})|D_{x}u^{\epsilon}-D_{x}u^{\epsilon'}|^{2}\leq M_{\epsilon,\epsilon'}+M_{\epsilon',\epsilon},
\]
which implies (\ref{eq: Du conv  L^2m}). Finally, sending $\epsilon'\rightarrow0$
first and then $\epsilon\rightarrow0$, using the fact that $K_{\epsilon',\epsilon}^{-}\rightarrow0$,
yields (\ref{eq:epsilon log m conv}).

Next we show the continuity properties of $m$, which in particular
give a meaning to the expression $m(\cdot,T)$. Integrating the equation
for $m^{\epsilon}-m^{\epsilon'}$ over a cylinder $\mathbb{T}^{d}\times[0,t]$
against a test function $\phi\in C^{\infty}(\mathbb{T}^{d})$ yields
\begin{align*}
\intO(m^{\epsilon}-m^{\epsilon'})\phi & =-\int_{0}^{t}\int_{\mathbb{T}^{d}}(m^{\epsilon}D_{p}H(\cdot,D_{x}u^{\epsilon})-m^{\epsilon'}D_{p}H(\cdot,D_{x}u^{\epsilon'}))\cdot D\phi,
\end{align*}
which implies
\[
\sup_{t\in[0,T]}||m^{\epsilon}(\cdot,t)-m^{\epsilon'}(\cdot,t)||_{H^{-1}(\mathbb{T}^{d})}\leq\sqrt{T}||m^{\epsilon}D_{p}H(\cdot,D_{x}u^{\epsilon})-m^{\epsilon'}D_{p}H(\cdot,D_{x}u^{\epsilon'})||_{L^{2}(Q_{T})}.
\]
Consequently, (\ref{eq: H>=00003D|p|^2-C}) and (\ref{eq: Du conv  L^2m})
together imply that $m^{\epsilon}\rightarrow m$ in $C^{0}([0,T],H^{-1}(\mathbb{T}^{d}))$.
Similarly, testing $\phi$ against the continuity equation for $m^{\epsilon}$
in a cylinder $\mathbb{T}^{d}\times[s,t]$ yields
\[
\intO(m^{\epsilon}(x,t)-m^{\epsilon}(x,s))\phi(x)dx=-\int_{s}^{t}\int_{\mathbb{T}^{d}}m^{\epsilon}D_{p}H(x,D_{x}u^{\epsilon})\cdot D\phi(x)dxdt,
\]
thus 
\[
||m^{\epsilon}(\cdot,t)-m^{\epsilon}(\cdot,s)||_{H^{-1}(\mathbb{T}^{d})}\leq\sqrt{t-s}||m^{\epsilon}D_{p}H(\cdot,D_{x}u^{\epsilon})||_{L^{2}(Q_{T})},
\]
and sending $\epsilon\rightarrow0$ produces the desired Hölder estimate.

Since $m^{\epsilon}(\cdot,T)\rightarrow m_{T}$ a.e. in $\mathbb{T}^{d}$,
and $m^{\epsilon}(\cdot,T)\rightarrow m(\cdot,T)$ in $H^{-1}(\mathbb{T}^{d})$,
it follows that $m_{T}=m(\cdot,T)$, and that $u^{\epsilon}(\cdot,T)=g(\cdot,m^{\epsilon}(\cdot,T))\rightarrow g(\cdot,m(\cdot,T))$
almost everywhere. Moreover, (\ref{eq:u conv in BV}) implies that
$u^{\epsilon}(\cdot,T)\weak u(\cdot,T)\text{ in }C^{\infty}(\mathbb{T}^{d})^{*}$,
and therefore $u(\cdot,T)=g(\cdot,m(\cdot,T))$.

It remains to show (\ref{eq:weak identity for all s}). For this purpose,
performing once more the Lasry--Lions computation, this time on the
system for $(u^{\epsilon},m^{\epsilon})$, and integrating on $\mathbb{T}^{d}\times[s,T]$
yields
\[
\int_{s}^{T}\intO m^{\epsilon}(x,t)(H^{\epsilon}-D_{p}H^{\epsilon}\cdot D_{x}u^{\epsilon}-f(x,m^{\epsilon})-\epsilon\log m^{\epsilon})dxdt=\int_{\mathbb{T}^{d}}(m^{\epsilon}(x,T)g(x,m^{\epsilon}(x,T))-m^{\epsilon}(x,s)u^{\epsilon}(x,s))dx,
\]
By Fubini's theorem, for a.e. $s\in[0,T]$, $u^{\epsilon}(\cdot,s)$
and $m^{\epsilon}(\cdot,s)$ converge, respectively, to $u(\cdot,s)\in L^{\infty}(\mathbb{T}^{d})$
and $m(\cdot,s)\in L^{\infty}(\mathbb{T}^{d})$ a.e. in $\mathbb{T}^{d}$.
Thus, for such $s$, using (\ref{eq: Du conv  L^2m}), one obtains
(\ref{eq:weak identity for all s}) after letting $\epsilon\rightarrow0$.
When $s=0$, $m^{\epsilon}(s)=m_{0}$, and thus the $C^{\infty}(\mathbb{T}^{d})^{*}$--convergence
of $u^{\epsilon}(0)$ is sufficient to conclude.
\end{proof}
We now prove the main result for the degenerate elliptic problem,
Theorem \ref{thm:weaksols}. 
\begin{proof}[Proof of Theorem \ref{thm:weaksols}]
First we will establish that $(u,m)$ is indeed a weak solution.
By Lemmas \ref{lem: weak Linfty estimates} and \ref{lem:convergence lemma weak},
$(u,m)$ satisfies condition (i) of Definition \ref{def:weaksol def}.
Next, by the HJ equation for $u^{\epsilon},$
\[
-u_{t}^{\epsilon}+H^{\epsilon}=f(\cdot,m^{\epsilon})+\epsilon\log m^{\epsilon}\leq f(\cdot,m^{\epsilon})+\epsilon(\log m^{\epsilon})^{+}.
\]
Integration against a non-negative function $\phi\in C^{\infty}(\overline{Q_{T}})$
then yields
\begin{equation}
\intO(-u^{\epsilon}(\cdot,T)\phi(\cdot,T)+u^{\epsilon}(\cdot,0)\phi(\cdot,0))+\intQ u^{\epsilon}\phi_{t}+\intQ H^{\epsilon}\phi\leq\intQ f(\cdot,m^{\epsilon})\phi+\epsilon(\log m^{\epsilon})^{+}\phi.\label{eq:inter weakHJ}
\end{equation}
Using the convexity of $H$,
\[
\intQ H(\cdot,D_{x}u)\phi+D_{p}H(\cdot,D_{x}u)\cdot(D_{x}u^{\epsilon}-D_{x}u)\phi\leq\intQ H(\cdot,D_{x}u^{\epsilon})\phi,
\]
and therefore, since $D_{x}u^{\epsilon}\rightharpoonup D_{x}u$ in
$L^{2}(Q_{T})$,
\[
\intQ H\phi\leq\liminf_{\epsilon\rightarrow0^{+}}\intQ H^{\epsilon}\phi.
\]
Letting $\epsilon\rightarrow0$ in (\ref{eq:inter weakHJ}), Lemmas
\ref{lem: weak Linfty estimates} and \ref{lem:convergence lemma weak}
and (\ref{eq:u conv in BV}) thus yield
\[
\intO(-u(\cdot,T)\phi(\cdot,T)+u(\cdot,0)\phi(\cdot,0))+\intQ u\phi_{t}+\intQ H\phi\leq\intQ f\phi.
\]
This completes the proof of condition (ii) in Definition \ref{def:weaksol def}.
The third condition follows immediately by testing the continuity
equation of $m^{\epsilon}$ against an arbitrary $\phi\in C^{\infty}(\overline{Q_{T}})$,
and appealing to Lemma \ref{lem:convergence lemma weak} to let $\epsilon\rightarrow0$
in the equality
\[
\intO(m^{\epsilon}(\cdot,T)\phi(\cdot,T)-m_{0}\phi(\cdot,0))-\intQ m^{\epsilon}\phi_{t}+\intQ m^{\epsilon}D_{p}H(\cdot,D_{x}u^{\epsilon})\cdot D_{x}\phi=0.
\]
Finally, condition (iv) of Definition \ref{def:weaksol def} has already
been established in Lemma \ref{lem:convergence lemma weak}.

Next is the proof of uniqueness. Let $(u',m')$ be another weak solution
to (\ref{eq:mfg}). By the fact that every non-negative distribution
can be identified with a non-negative measure, the HJ inequality for
$u'$ may be written as
\[
-u'_{t}+H(\cdot,D_{x}u')+\mu=f(\cdot,m'),
\]
where $\mu\in C(\overline{Q_{T}})^{*}$ is a non-negative, finite
measure on $\overline{Q_{T}}$. We carefully apply the Lasry-Lions
procedure to the $(u^{\epsilon},m^{\epsilon})$ and $(u',m')$ systems.
Set $v^{\epsilon}=u^{\epsilon}-u'$, $w^{\epsilon}=m^{\epsilon}-m'.$
Subtracting the two corresponding HJ equations for $u',u^{\epsilon}$,
and integrating against $m^{\epsilon}$, we obtain
\begin{equation}
\intQ(v_{t}^{\epsilon}m^{\epsilon}+(H'-H^{\epsilon})m^{\epsilon})+\intQ m^{\epsilon}d\mu=\intQ m^{\epsilon}(f'-f^{\epsilon}),\label{eq:1 weak}
\end{equation}
where $H'=H(\cdot,D_{x}u')$, $f'=f(\cdot,m')$. Integrating the HJ
equation for $u^{\epsilon}$ against $m'$ gives
\[
\intQ(-u_{t}^{\epsilon}m'+H^{\epsilon}m')=\intQ m'f^{\epsilon}.
\]
Now, subtracting the continuity equations for $m^{\epsilon},m'$ and
testing against $u^{\epsilon}$ yields
\begin{equation}
\intQ(-u_{t}^{\epsilon}m^{\epsilon}+u_{t}^{\epsilon}m')+(m^{\epsilon}D_{p}H^{\epsilon}-m'D_{p}H')\cdot D_{x}u^{\epsilon}=-\intO v^{\epsilon}(T)m'(T),\label{eq:2 weak}
\end{equation}
and integrating the continuity equation for $m^{\epsilon}$ against
$-u'$ we get
\begin{equation}
\intQ(m^{\epsilon}u'_{t}-m^{\epsilon}D_{p}H^{\epsilon}D_{x}u')=\intO(m^{\epsilon}(T)u'(T)-m_{0}u'(0)).\label{eq:3 weak}
\end{equation}
Finally, by Condition (iv) in Definition \ref{def:weaksol def} and
the fact that $(u',m')$ is a weak solution,

\begin{equation}
\intQ-m'(H'-D_{p}H'\cdot D_{x}u')=-\intQ m'f'-\int_{\mathbb{T}^{d}}(m'(T)u'(T)-m_{0}u'(0)).\label{eq:4 weak}
\end{equation}
Adding (\ref{eq:1 weak}), (\ref{eq:2 weak}), (\ref{eq:3 weak}),
and (\ref{eq:4 weak}) yields
\begin{equation}
M_{\epsilon,1}+M_{\epsilon,2}+M_{g,\epsilon}+M_{f,\epsilon}+K_{\epsilon}^{-}+K_{1,\epsilon}=K_{2,\epsilon}+K_{\epsilon}^{+},\label{eq: final weak}
\end{equation}
where
\[
M_{\epsilon,1}=\intQ m^{\epsilon}(H'-H^{\epsilon}-D_{p}H'\cdot(D_{x}u^{\epsilon}-D_{x}u')),\;M_{\epsilon,2}=\intQ m'(H^{\epsilon}-H'-D_{p}H^{\epsilon}\cdot(D_{x}u'-D_{x}u^{\epsilon})),
\]
\[
M_{g,\epsilon}=\intO(m^{\epsilon}(\cdot,T)-m'(\cdot,T))(g(\cdot,m^{\epsilon}(\cdot,T))-g(\cdot,m'(\cdot,T))),\;M_{f,\epsilon}=\intQ(f(\cdot,m^{\epsilon})-f(\cdot,m'))(m^{\epsilon}-m'),
\]
\[
K_{\epsilon}^{\pm}=\intQ\epsilon(\log m^{\epsilon})^{\pm}m',\;K_{1,\epsilon}=\intQ m^{\epsilon}d\mu,\;K_{2,\epsilon}=-\intQ\epsilon\log m^{\epsilon}m^{\epsilon}.
\]
The only new term relative to the proof of Lemma \ref{lem:convergence lemma weak}
is $K_{1,\epsilon}$, which is clearly non-negative. Thus, as before,
each individual term on the left-hand side of (\ref{eq: final weak})
converges to zero as $\epsilon\rightarrow0$. In particular, since
$M_{f,\epsilon}\rightarrow0$, it follows that $m=m'$ a.e.. A posteriori,
since $M_{\epsilon,2}\rightarrow0$, (\ref{eq: Du conv  L^2m}) and
the strict convexity of $H$ imply that $D_{x}u=D_{x}u'$ a.e. in
$\{m>0\}.$ Moreover, $M_{g,\epsilon}\rightarrow0$ implies $m(\cdot,T)=m'(\cdot,T)$
a.e., and thus $u(\cdot,T)=g(\cdot,m(\cdot,T))=g(\cdot,m'(\cdot,T))=u'(\cdot,T)$
a.e. in $\mathbb{T}^{d}$.

It remains to show that $u(s)=u'(s)$ a.e. in $\{m(s)>0\}$, for a.e.
$s\in[0,T]$ including the case $s=0$. The function $\overline{u}=\max(u,u')$
is in $\text{{BV}}(Q_{T})\cap L^{\infty}(Q_{T})$, and, since $D_{x}u,D_{x}u'\in L^{2}(Q_{T})$,
the chain rule yields $D_{x}\overline{u}=D_{x}u$ a.e. in $\{u\geq u'\}$
and $D_{x}\overline{u}=D_{x}u'$ a.e. in $\{u\leq u'\}$. In particular,
$D_{x}\overline{u}=D_{x}u$ a.e. in $\{m>0\}$. Following \cite[Thm 5.2]{CardaliaguetGraberPorrettaTonon},
through a mollification procedure, the theory of viscosity solutions
implies that $\overline{u}$ is a distributional subsolution to the
HJ equation. Therefore, testing the HJ inequality of $\overline{u}$
against $m^{\epsilon}$ in an interval $[s,T]$,
\begin{align*}
\int\overline{u}(s)m^{\epsilon}(s)\leq & \intO\overline{u}(T)m^{\epsilon}(T)+\int_{s}^{T}\intO(-m_{t}^{\epsilon}\overline{u}-m^{\epsilon}H(x,D_{x}\overline{u}))+m^{\epsilon}f(x,m)dxdt\\
= & \intO u(T)m^{\epsilon}(T)+\int_{s}^{T}\intO(-\text{div}(m^{\epsilon}D_{p}H^{\text{\ensuremath{\epsilon}}})\overline{u}-m^{\epsilon}H(x,D_{x}\overline{u})+m^{\epsilon}f(x,m))dxdt\\
= & \intO u(T)m^{\epsilon}(T)+\int_{s}^{T}\intO m^{\epsilon}(D_{p}H^{\epsilon}\cdot D_{x}\overline{u}-H(x,D_{x}\overline{u})+f(x,m))dxdt.
\end{align*}
Using Lemma \ref{lem:convergence lemma weak} and the dominated convergence
theorem to let $\epsilon\rightarrow0$, 
\begin{align*}
\intO\overline{u}(s)m(s) & \leq\intO u(T)m(T)+\int_{s}^{T}\intO m(D_{p}H\cdot D_{x}\overline{u}-H(x,D_{x}\overline{u})+f(x,m))dxdt\\
 & =\intO u(T)m(T)+\int_{s}^{T}\intO m(D_{p}H\cdot D_{x}u-H(x,D_{x}u)+f(x,m))dxdt\\
 & =\intO u(s)m(s)
\end{align*}
for a.e. $s\in[0,T]$, including $s=0$, where (\ref{eq:weak identity for all s})
was used in the last equality. Given that $\overline{u}\geq u$, this
implies $\overline{u}(s)=u(s)$ a.e. in $\{m(s)>0\}$ for such $s$.
At last, applying the Lasry-Lions method to the system for $(u^{\epsilon}-u',$
$m^{\epsilon}$), over an interval $[s,T]$, yields
\begin{align*}
\int_{s}^{T}\int_{\mathbb{T}^{d}}\left((H'-H^{\epsilon})m^{\epsilon}+m^{\epsilon}D_{p}H^{\epsilon}\cdot(D_{x}u^{\epsilon}-D_{x}u')\right)+\int_{s}^{T}\intO m^{\epsilon}d\mu\\
+\int_{s}^{T}\intO m^{\epsilon}(f^{\epsilon}-f')+\intO\left(m^{\epsilon}(T)v^{\epsilon}(T)-m^{\epsilon}(s)v^{\epsilon}(s)\right)=0.
\end{align*}
Since $\int_{s}^{T}\intO m^{\epsilon}d\mu\leq\intQ m^{\epsilon}d\mu=K_{1}\rightarrow0$,
sending $\epsilon\rightarrow0$ results in\vspace*{-0.1cm}
\[
\intO m(T)(u(T)-u(T))-m(s)(u(s)-u'(s))=0,
\]
that is,\vspace*{-0.15cm}
\[
\intO m(s)u(s)=\intO m(s)u'(s).
\]
By the fact that $u(s)=\overline{u}(s)\geq u'(s)$ a.e. in $\{m(s)>0\}$,
one then concludes that $u=u'$ a.e. in $\{m>0\}$ and, since $m_{0}>0$,
we have $u(0)=u'(0$) a.e. in $\mathbb{T}^{d}.$ 
\end{proof}
Next, we prove that when the data is independent of the space variable,
the solution $u$ is Lipschitz continuous.
\begin{thm}
\label{thm:Lip degenerate}Assume that (\ref{eq:DE}) holds, and let
$H$, $f$, and $g$ be independent of $x$. Then the sequence $u^{\epsilon}$
is uniformly bounded in $C^{1}(\overline{Q_{T}})$ as $\epsilon\rightarrow0$.
In particular, the weak solution $u=\lim_{\epsilon\rightarrow0}u^{\epsilon}$
and the terminal condition $m(\cdot,T)$ are globally Lipschitz continuous.
\end{thm}

\begin{proof}
There is no loss of generality in assuming $f(0)<-1$. By Remark \ref{rem:remark u, m(T) bound},
\begin{equation}
\min m_{0}\leq m^{\epsilon}(\cdot,T)\leq\max m_{0}.\label{eq:Thm 5.5 density bounds}
\end{equation}
 We set, for $K>0$,
\[
\chi^{\epsilon}(w)=(f^{\epsilon})^{-1}(w)\cdot f_{m}^{\epsilon}((f^{\epsilon})^{-1}(w)),\;h^{\epsilon}=\sqrt{\chi^{\epsilon}},
\]
\[
\beta_{K}^{\epsilon}=||f^{\epsilon}||_{C^{1}([\frac{1}{K},K])}+||g||_{C^{1}([\frac{1}{K},K])}+\left\Vert \frac{1}{\chi^{\epsilon}}\right\Vert _{C^{0}([-K,\infty))}+||h_{w}^{\epsilon}||_{C^{0}([-K,\infty))}.
\]
Then, in view of (\ref{eq:DE}), (\ref{eq:Thm 5.5 density bounds}),
and Lemma \ref{lem: gradient a priori bound}, there exist constants
$C$ and $C_{1}$, with $C=C(C_{1},\beta_{C_{1}}^{\epsilon})$ and
$C_{1}=C_{1}(C_{0},T,T^{-1},(1-\tau)^{-1},||u^{\epsilon}||_{C^{0}(\QT)})$,
such that
\[
||Du^{\epsilon}||_{C^{0}(\QT)}\leq C.
\]
The only issue here is that the quantities
\[
K_{1,\epsilon}(-C_{1})=\left\Vert \frac{1}{\chi^{\epsilon}}\right\Vert _{C^{0}([-C_{1},\infty))}\text{ and \;\;\;}K_{2,\epsilon}(-C_{1})=||h_{w}^{\epsilon}||_{C^{0}([-C_{1},\infty))}
\]
may not be bounded independently of $\epsilon$. However, the proof
of Lemma \ref{lem: gradient a priori bound} shows that, defining
$\utilde^{\epsilon}$ as in said proof, the gradient bound depends
only on $K_{1,\epsilon}(a)$ and $K_{2,\epsilon}(a)$, where $a\in\mathbb{R}$
is any number satisfying the following condition: for all small enough
$0\leq c_{1}<1,$ at any maximum point $(x_{0},t_{0})$ of $H(D_{x}u^{\epsilon})+\frac{c_{1}}{2}(\utilde^{\epsilon})^{2},$
the inequality $f^{\epsilon}(m(x_{0},t_{0}))\geq a$ holds. At such
a point $(x_{0},t_{0})$, for small enough $\epsilon$ and $c_{1}$,
Corollary \ref{cor:utbound} yields
\begin{multline*}
f^{\epsilon}(m(x_{0},t_{0}))=-u_{t}^{\epsilon}+H(D_{x}u^{\epsilon})\geq f^{\epsilon}(\min m_{0})-||H(D_{x}u^{\epsilon})||_{C^{0}}+H(D_{x}u^{\epsilon})\\
\geq f^{\epsilon}(\min m_{0})-\frac{c_{1}}{2}||\utilde^{\epsilon}||_{C^{0}}^{2}=f(\min m_{0})+\epsilon\log\min m_{0}-c_{1}||\utilde^{\epsilon}||_{C^{0}}^{2}>f(\frac{1}{2}\min m_{0}).
\end{multline*}
Thus, the condition is satisfied by $a=f(\frac{1}{2}\min m_{0})$.
Since $a>f(0),$ it follows from (\ref{eq:f polynomial growth}),
(\ref{growth chi}) that $K_{1,\epsilon}(a)$ and $K_{2,\epsilon}(a)$
are bounded uniformly as $\epsilon\rightarrow0$. The Arzelà--Ascoli
Theorem implies the result.
\end{proof}
We finally note that, in the case $d=1$, since there exists an a
priori lower bound for the density $m$ in terms of its boundary values
(obtained in \cite{BakaryanFerreiraGomes,Gomes,LavenantSantambrogio}),
the solutions are seen to be smooth. However, when $d>1$, even in
the special case of Theorem \ref{thm:Lip degenerate}, where an a
priori bound for the gradient was obtained, we do not know whether
the solution to the degenerate elliptic problem enjoys higher regularity. 

\subsection*{Acknowledgements}

The author was partially supported by P.E. Souganidis's National Science
Foundation grant DMS-1900599, the Office for Naval Research grant
N000141712095 and the Air Force Office for Scientific Research grant
FA9550-18-1-0494. The author would like to thank P.E. Souganidis for
valuable discussions, comments, and suggestions. He also thanks the
anonymous referees for their careful reading and help in improving
and clarifying the manuscript. 

{{\footnotesize   \bigskip   \footnotesize   \textsc{Department of Mathematics, University of Chicago, Illinois, 60637, USA}\\ \nopagebreak \textit{E-mail address}: \texttt{sbstn@math.uchicago.edu}}}

\begin{thebibliography}{10}
\bibitem{BakaryanFerreiraGomes}T. Bakaryan, R. Ferreira, D.A. Gomes,
Uniform estimates for the planning problem with potential, NoDEA Nonlinear
Differential Equations Appl. 28 (2021), no. 2, Paper No. 20, 23.

\bibitem{Cardalaguiet}P. Cardaliaguet, Weak solutions for first order
mean field games with local coupling, Analysis and geometry in control
theory and its applications, Springer INdAM Ser. 11 (2015) 111-158.

\bibitem{CardaliaguetGraber}P. Cardaliaguet, P.J. Graber, Mean field
games systems of first order, ESAIM: Contr. Opt. and Calc. Var., 21
(3) (2015) 690-722.

\bibitem{CardaliaguetGraberPorrettaTonon}P. Cardaliaguet, P.J. Graber,
A. Porretta, D. Tonon, Second order mean field games with degenerate
diffusion and local coupling, Nonlinear Differ. Equ. Appl. 22 (2015)
1287--1317.

\bibitem{CardaliaguetPorretta}P. Cardaliaguet, A. Porretta, An introduction
to mean field game theory. In Mean Field Games, pp. 1-158, Lecture
Notes in Math. 2281, Springer, Cham, 2020.

\bibitem{Evans}L.C. Evans, Some new PDE methods for weak KAM theory,
Calculus of Variations and Partial Differential Equations, 17 (2)
(2003) 159--177.

\bibitem{Fiorenza}R. Fiorenza, Sui problemi di derivata obliqua per
le equazioni ellittiche, Ric. Mat. 8 (1959), 83-110.

\bibitem{Fiorenza-1} R. Fiorenza, Sulla hölderianità della soluzioni
dei problemi di derivata obliqua regolare del secondo ordine, Ric.
Mat. 14 (1965) 102-123.

\bibitem{GilbargTrudinger}D. Gilbarg, N.S. Trudinger, Elliptic partial
differential equations of second order, Springer, Berlin, 2001, pp.
120-130.

\bibitem{Gomes}D.A. Gomes, T. Seneci, Displacement convexity for
first-order Mean-Field Games, Minimax Theory Appl. 3 (2018), no. 2,
261--284.

\bibitem{Gomes-1}D.A. Gomes, H. Mitake, K. Terai, The selection problem
for some first-order stationary Mean Field Games, Netw. Heterog. Media
15 (2020), no. 4, 681--710.

\bibitem{Graber}P.J. Graber, Optimal Control of first-order Hamilton-Jacobi
equations with linearly bounded Hamiltonian, Applied Mathematics and
Optimization 70 (2013) 185-224.

\bibitem{caines}M. Huang, R.P. Malhamé, P.E. Caines, Large population
stochastic dynamic games: closed-loop McKean-Vlasov systems and the
Nash certainty equivalence principle, Comm. Inf. Syst. 6 (2006) 221--251.

\bibitem{LasryLions}J.-M. Lasry, P.-L. Lions, Mean field games, Jpn.
J. Math. 2 (1) (2007) 229--260.

\bibitem{LavenantSantambrogio}H. Lavenant, F. Santambrogio, Optimal
density evolution with congestion: $L^{\infty}$ bounds via flow interchange
techniques and applications to variational Mean Field Games, Comm.
Partial Differential Equations 43 (12) (2018) 1761-1802.

\bibitem{Lieberman} G.M. Lieberman, The nonlinear oblique derivative
problem for quasilinear elliptic equations, Non-linear analysis, Theory,
Methods \& Applications (1984).

\bibitem{Lieberman solvability}G.M. Lieberman, Solvability of quasilinear
elliptic equations with nonlinear boundary conditions, Trans. Amer.
Math. Soc. 273 (1982) 753-765.

\bibitem{Lions} P.-L. Lions, Courses at the Collège de France, www.college-de-france.fr

\bibitem{LionsSoug}P.-L. Lions, P.E. Souganidis, Extended Mean-Field
Games, Atti Accad. Naz. Lincei Cl. Sci. Fis. Mat. Natur. 31 (2020),
611-625.

\bibitem{Munoz}S. Munoz, Classical solutions to local first order
Extended Mean Field Games, arXiv:2102.13093 (2021).

\bibitem{Porretta} A. Porretta, Weak Solutions to Fokker-Planck Equations
and Mean Field Games, Arch. Rational Mech. Anal. 216, 1-62 (2015).
\end{thebibliography}
\end{document}